\documentclass[final,onefignum,onetabnum]{siamart220329}       
\usepackage{amsmath,amsfonts,amssymb}%AMS数学符号字体等的宏包
                       
\usepackage{graphicx}                          %插图
\usepackage{setspace}  
\usepackage{caption}
\usepackage{listings}
\usepackage{tabularx}
\usepackage{xcolor}
\usepackage{indentfirst}
\usepackage{cite}
\usepackage{subcaption}

\usepackage{multirow}
\usepackage{overpic}

\usepackage[normalem]{ulem}

\usepackage{verbatim}

\usepackage{enumitem}

\usepackage{tikz}

\usepackage{bm}

\theoremstyle{definition}

\theoremstyle{remark}
\newtheorem{remark}[theorem]{Remark}

\DeclareMathOperator{\tr}{tr}

\newcommand{\x}{{\bm x}}
\newcommand{\dd}{\mathrm{d}}

\numberwithin{equation}{section}

\begin{document}
% \title{The Non-Isothermal Allen-Cahn Equations:\\
% Derivation, Asymptotic Limits and Simulations}
\title{Non-isothermal diffuse interface model for phase transition and interface evolution}

%DIFFUSIVE INTERFACE MODEL FOR INTERFACE EVOLUTIONS IN NON-ISOTHERMAL FLUIDS

%  Temperature effects with diffuse interface method: Allen-Cahn, curvature effects in Stefan problem
\author{Chun Liu\thanks{Department of Applied Mathematics, Illinois Institute of Technology, Chicago, IL. Email: cliu124@iit.edu} \and Jan-Eric Sulzbach\thanks{Department of Mathematics, Technical University Munich, Munich, Germany. Email: janeric.sulzbach@tum.de}\and  Yiwei Wang\thanks{Department of Mathematics, University of California, Riverside, CA.~ Email: yiweiw@ucr.edu}, corresponding author }
 
\date{}
\maketitle
\begin{abstract}
In this paper, we derive a thermodynamically consistent non-isothermal diffuse interface model for  phase transition and interface evolution involving heat transfer. 
This model is constructed by integrating concepts from classical irreversible thermodynamics with the energetic variational approach. By making specific assumptions about the kinematics of the temperature, we derive a non-isothermal Allen-Cahn equation.
Through both asymptotic analysis and numerical simulations, we demonstrate that in the sharp interface limit, the non-isothermal Allen-Cahn equation converges to a two-phase 
nonlinear Stefan type problem, under a certain scaling of the melting/freezing energy.
In this regime, the motion of the liquid-solid interface and the temperature interface coincide and are governed by the mean curvature, at least for a short time. The result provides a justification for the classical Stefan problem within a certain physical regime.
\end{abstract}

\section{Introduction}

Understanding the evolution of the interface between two phases undergoing a phase transition is an important problem in both physics and mathematics \cite{alexiades2018mathematical}. 
One classical approach to studying this behavior is through Stefan problems \cite{rubinshteuin1971stefan, stefan1891theorie, visintin2008introduction}.
Generally speaking, the formulation of a Stefan problem involves solving a heat equation within each respective phase, while adhering to prescribed initial and boundary conditions. At the interface, where the two phases meet, the temperature is fixed at the critical temperature for phase transition, which is a key feature of classical Stefan problems. To complete the mathematical framework, an additional equation known as the Stefan condition is imposed. This condition, rooted in energy conservation principles, serves to precisely determine the position of the advancing interface.
Stefan problems have been one of the most well-studied classes of free boundary problems dating back over a hundred years with the pioneering work by  Lame \cite{lame1831memoire} and Stefan \cite{stefan1891theorie} and more recent results by \cite{rubinshteuin1971stefan,pruss2007existence,bernauer2011optimal,mamode2013two,pegler2021convective}.
Key results such as the existence of solutions , the continuity of the moving boundary and the regularity of solutions can be found in a series of papers \cite{caffarelli1977regularity,kinderlehrer1978smoothness,caffarelli1979continuity,caffarelli1983continuity,athanasopoulos1996regularity, escher1996analyticity}.
However, it remains unclear whether such a simple model is truly valid in reality. For instance, classical Stefan problems typically consider only a single interface by fixing the temperature at the phase change temperature at the liquid-solid interface, which is not true in general. 
Whereas in reality, it is often observed that the liquid phase is below its freezing point or the solid phase is above it, which is known as supercooling and  superheating \cite{caginalp1986analysis, sherman1970general,dewynne1990survey,baker2022zero,cuchiero2023propagation}.
More importantly, one could ask what is the thermodynamic reasoning behind Stefan problems? Are these models thermodynamically consistent, i.e., do they satisfy the first and second laws of thermodynamics?

In this paper, we revisit this classical problem  by developing a new thermodynamically consistent non-isothermal diffuse interface model for  phase transition and interface evolution involving heat transfer.
In contrast to the approaches in Stefan problems, which assume that there exists 
a sharp interface between the liquid and solid phase, diffuse interface models \cite{allen1979microscopic} assume that there exists an interfacial region between the two phases, which is closer to the physical reality \cite{caginalp1986analysis}. One can introduce a  smooth space-time-dependent phase function $\phi(\x, t)$ to identify each phase, as well as the interface.
Classical diffuse interface models often consider isothermal cases \cite{allen1979microscopic, cahn1958free}.
However, to model phase transitions and heat transfer as considered in the Stefan problem \cite{Landau2013collected}, it is crucial to include the temperature as an additional variable, which is challenging because one requires the system to be both physically and mathematically consistent.
Some advances in this direction have been made by Caginalp in a series of papers \cite{caginalp1986analysis,caginalp1990dynamics,caginalp1995derivation,caginalp1998convergence}, who was one of the first to study temperature dependent phase field models and who found the relation of these systems with free boundary problems, in particular the Stefan problem, as their sharp interface limit \cite{caginalp1998convergence}. Around the same time, Penrose and Fife introduced their famous model, where both the phase parameter and the temperature can depend on space and time \cite{penrose1990thermodynamically}. The connection between the Penrose-Fife model and Stefan-type problems has been explored in the literature \cite{colli1997stefan, colli1995penrose}.
However, in both models, the temperature equation only contains linear terms in both the phase variable and temperature variable.
In this work we take a new approach by combining ideas from the energetic variational approach with non-equilibrium thermodynamics. This approach works with the entropy of the system rather than the enthalpy as main thermodynamic quantity as most of the previous work, which allows us to include higher order nonlinearities of the system that could get lost in other cases.
Furthermore, our approach allows for various assumptions regarding the kinematics of the temperature, corresponding to different heat transfer mechanisms in physical bodies. This flexibility enables applications across a wide range of physical models.

By assuming the  temperature is on a fixed background and not transported with material particles, 
we derive a model that generalizes the classical Allen-Cahn equation to a non-isothermal setting. We then investigate the sharp interface limit of this non-isothermal Allen-Cahn equation, which validates the classical Stefan problem within a specific physical regime.
Through asymptotic analysis, we show that under certain scaling of the melting/freezing energy, the motion of the liquid-solid interface and the temperature interface coincide and are governed by its mean curvature. We also conduct numerical simulations of the diffuse interface model, which demonstrate that the motion of both interfaces follows the mean curvature flow, at least for a short time. The simulation results further support the asymptotic analysis and indicate that the Stefan model is a good reduced model within this regime.

The rest of paper is organized as follows: In Section \ref{sect: Derivation}, we derive, following a free energy approach, non-isothermal diffuse interface models, including the non-isothermal Allen-Cahn equation, under various assumptions on the kinematics of the temperature. 
In Section \ref{sect: Limit}, we show that the formal asymptotic limit of the non-isothermal Allen-Cahn system, under certain scaling of the melting/freezing energy, leads to a nonlinear two-phase Stefan type problem.
Hence, the Stefan problem can be seen as a reduced model of the non-isothermal Allen-Cahn equation.  Finally, we perform some numerical studies to the non-isothermal Allen-Cahn model 
in Section \ref{sect: Numerics}.

\section{Derivation of non-isothermal diffuse interface models}\label{sect: Derivation}

In this section, we derive several non-isothermal diffuse interface models, including the non-isothermal Allen-Cahn equation, for phase transition between a liquid and solid phase. These models are developed by combining ideas from the classical Energetic Variational Approach (EnVarA) for isothermal system \cite{giga2017variational,wang2022some} and non-equilibrium thermodynamics \cite{Fremond, coleman1967thermodynamics}. 

We first present a practical user manual designed for developing a wide range of non-isothermal physical models:
\begin{itemize}
    \item Formulate the free energy density $\psi(f, \theta)$ of the system, which includes both mechanical (resolved) and thermal (unresolved) contributions. 
    Here $f$ denotes the mechanical state variable, such as density and phase function, and $\theta$ represents the temperature.
    From this free energy, one can derive other thermodynamic quantities, such as entropy $s(f, \theta)$ and internal energy $e(f, \theta)$, according to the fundamental thermodynamic relations.

    \item Specify the kinematics of the state variable $f$ and the temperature $\theta$, i.e, describes how $f$ and $\theta$ change when the material particles move, without microscopic evolution or other thermal processes.

    \item Derive the conservative and dissipative forces by using the Energetic Variational Approach (with prescribed mechanical dissipation) and combine them using a force balance condition, which leads to the equation of $f$.
    \item Determine the equation of the temperature $\theta$ or entropy $s(f, \theta)$ by using the laws of thermodynamics and constitutive relations, including the Clausius-Duhem relation and Fourier's law.
\end{itemize}

This machinery has been successfully used to model other temperature-dependent fluid systems such as, the non-isothermal liquid crystal flow \cite{de2019non},  the Brinkmann-Fourier system \cite{Liu2022}, a chemical reaction diffusion system \cite{liu2022well} and the non-isothermal Cahn-Hilliard model \cite{de2024temperature}. 
We will explain each step in the following, focusing on building non-isothermal diffuse interface models for phase transitions.

\subsection{Free energy formalism of diffuse interface models}\label{sec: free energy}

To model phase transitions, we start with a modified Ginzburg-Landau free energy density $\psi(\phi,\theta)$, which includes additional terms to account for freezing/melting and thermal energy contributions, given by
\begin{align}\label{eq: free energy Allen Cahn}
\psi(\phi,\theta)= \frac{\epsilon}{2}|\nabla \phi|^2+
\frac{1}{\epsilon} W(\phi) -\chi(\epsilon)\frac{\lambda}{\theta_c}(\theta-\theta_c)g(\phi) + c_v\theta(1-\log\theta).
\end{align}
Here, $\phi$ is the phase function, $\theta$ is the temperature, $\theta_c$ is the critical temperature at which the system undergoes the phase transition, $\epsilon$ is the small parameter related to the thickness of the interface, and $W(\phi)$ is the interfacial potential energy.
The parameter $c_v$ corresponds to the volumetric heat capacity, while $\chi(\epsilon) \lambda$ represents the parameters related to the latent heat of the system.
We use $\chi(\epsilon)$ to denote the scaling of the latent heat with respect to the interfacial parameter $\epsilon$, which plays an important role in the following analysis.
% ({\color{red}What is $\chi(\epsilon)$})
We refer to \cite{mcfadden1993phase,ciarletta2011phase} for details on the physical context of these parameters.
It is worth mentioning that $\phi$ is a dimensionless variable that labels the two phases, and the free energy (\ref{eq: free energy Allen Cahn}) is written in a non-dimensionalized form.

In classical diffuse interface models,  $\phi = \pm 1$ is often used to identify the solid and liquid phases. Hence, common choices for the interfacial potential energy are the logarithmic potential
\begin{align*}
    W(\phi)_{log}=(1 + \phi) \log(1 + \phi) + (1 - \phi) \log(1 - \phi) -  \phi^2,\quad \phi\in [-1,1],
\end{align*}
and the double-well potential
\begin{align*}
    W(\phi)_{dw}=\frac{1}{4}(\phi^2-1)^2,\quad \phi\in \mathbb{R}\ .
\end{align*}
The double-well potential can be seen as an approximation of the logarithmic potential.
Both potentials have two global minima in the interior of the physically relevant domain $[-1, 1]$, which is coherent with the physical interpretation of a diffuse interface model in which only the values of the variable $\phi \in [-1, 1]$ are meaningful.
We will work with the double-well potential throughout this paper.

\begin{remark}
We observe that for $\theta=\theta_c$, minimizing the free energy functional $\psi(\phi)$ corresponds to minimizing its two components, $W(\phi)$ and $|\nabla \phi|^2$. Therefore, candidates for the minimizers of \eqref{eq: free energy Allen Cahn} tend to take values at one of the two minima of the potential $W(\phi)$ while also having few oscillations between these two states to keep the gradient term small.
These regions of minimal energy are separated by a small interface, whose thickness is proportional to the parameter $\epsilon$. We will see in the following sections how these ideas transfer to the non-isothermal setting.
\end{remark}

Taking a closer look at the two additional terms in the free energy, when compared to the isothermal case, we observe that the third term represents the melting/freezing contribution to the free energy. 
The idea is that for temperatures $\theta > \theta_c$ we assume that the joint potential of $W(\phi)- \chi(\epsilon)\frac{\lambda}{\theta_c}(\theta - \theta_c)g(\phi) $ is no longer symmetric and that the local minimum for $\phi > 0$ is smaller than the local minima for $\phi < 0$. Thus, the global minimizer of the free energy is attained only for the state $\phi>0$.  
Similarly, we require a reverse statement to be true for $\theta<\theta_c$. Then, the temperature moves the free energy in favor of one phase state.
Two choices of $g(\phi)$ that satisfy the above are
\begin{align*}
    g(\phi)=\phi, \quad \textnormal{or } ~g(\phi)=-\frac{1}{3}\phi^3+\phi.
\end{align*}
The former is used in e.g. \cite{caginalp1986analysis}, whereas the advantage of the latter one is that the local minima of the joint potential are still attained at $\phi=\pm 1$.

\begin{remark}
The coefficient $\chi(\epsilon)$ allows for a different scaling of the melting/solidification term.
Common choices are $\chi(\epsilon)=1$. 
Other choices are $\chi(\epsilon)=\epsilon$ and $\chi(\epsilon)=\epsilon^{-1}$.
The effect of the later one is that now both the potential $W(\phi)$ and the melting/freezing term have the same scaling and thus can change the interface condition.
%For more details on the choice of $\chi(\epsilon)$ and the interplay with the nonlinearity $g(\phi)$ we refer to \cite{mcfadden1993phase,echebarria2004quantitative,wang2021phase}.
\end{remark}

The last contribution depends uniquely on the (absolute) temperature. This term is common in thermodynamics and is typical of the dynamics of the ideal gas, as we assume the free energy is concave with respect to the temperature variable.
This comes from the physical requirement that the heat capacity $c_v$ defined by $c_v = -\theta\partial_{\theta,\theta}\psi(\rho,\theta)$
is strictly positive.

\begin{remark}
    In the current study, we start with a simple form of the free energy to model phase transitions. Additional terms in the free energy and further assumptions such as a temperature dependent phase parameter $\epsilon=\epsilon(\theta)$ are also possible and have been studied for the case of the Cahn-Hilliard equation in \cite{de2024temperature}.
\end{remark}

After determining the form of the free energy density, one can define the entropy density $s(\phi, \theta)$ by
\begin{align}\label{eq: entropy Allen Cahn} 
s(\phi,\theta)= -\frac{\partial \psi}{\partial \theta}= \chi(\epsilon)\frac{\lambda}{\theta_c}g(\phi)+ c_v \log \theta.
\end{align}
From a thermodynamics point of view, the difference in the entropy between two states is responsible for the irreversibility of thermodynamic processes.
Moreover, we note that depending on $\chi(\epsilon)$ a separation of scales can also persist in the entropy.
In addition, we observe that for a fixed temperature $\theta$ the minimal entropy is obtained at the global minima of $g(\phi)$ and thus should be the preferred long-time state of the system with a slow evolution of the entropy. \\

Another quantity that we can define is the chemical potential $\mu=\mu(\phi,\theta)$, which is given by
\begin{align}\label{eq: chemical potential}
    \mu=\nabla\cdot\frac{\partial \psi(\phi,\theta)}{\partial\nabla\phi} -\frac{\partial\psi(\phi,\theta)}{\partial\phi}.
\end{align}
For the free energy (\ref{eq: free energy Allen Cahn}), we have
\begin{align}\label{eq:chem potential Allen Cahn}
    \mu(\phi,\theta)= \epsilon\Delta \phi-\frac{1}{\epsilon} W'(\phi)-\chi(\epsilon)\frac{\lambda}{\theta_c}(\theta-\theta_c)g'(\phi).
\end{align}

\subsection{Kinematics of the phase function and the temperature }\label{sec: kinematics of temp}
Now, we introduce the kinematics of the phase function $\phi$ and the temperature $\theta$. We assume that the media occupies a domain $\Omega \subseteq \mathbb{R}^d$, for $d=2,3$, with or without a boundary $\partial \Omega$.

In the setting of modeling phase transitions, $\phi$ is a non-conserved quantity, and we can impose the kinematics of the phase function $\phi$ as
\begin{align}\label{eq: kinematics}\begin{split}
\partial_t\phi +u \cdot \nabla \phi&=0,~~ (x,t)\in \Omega\times(0,T),\\
\phi_{|t=0}&=\phi_0,\\
 u\cdot n&=0, ~~(x,t)\in \partial \Omega\times(0,T),
\end{split}
\end{align}
supplemented by initial data and a no-flux boundary condition. Here, $u:\Omega \times (0,t)\to \mathbb{R}^d$ stands for the velocity of material points (or particles). Equation \eqref{eq: kinematics} is known as the scalar transport equation, which means that the information carried by the particle does not change when the particles move  \cite{liu2020variational}. More precisely, by assuming that $u$ has sufficient regularity, we can define the flow map $x(X, t)$ through
\begin{equation}\label{flow-map}\begin{split}
        \partial_t x(X,t)& = u(x(X,t),t),\quad (X,t) \in \Omega \times (0,T),\\
        x(X,0) &= X,\quad X\in \Omega,
        \end{split}
\end{equation}
where $X \in \Omega$ are the initial positions of the particles, known as the Lagrangian coordinates, and $x(X, t) \in \Omega$ are the Eulerian coordinates. 
 %throughout the paper we denote by $x\in \Omega$ the Eulerian coordinates and by $X\in \Omega$ the Lagrangian coordinates. 
% the transport equation for $\phi$ generates a smooth flow map
Then, the exact solution of \eqref{eq: kinematics} for the phase function $\phi$ is given by $\phi(t,x(X,t)) = \phi_0(X)$.

Similar to the mechanical variable, we shall ask 
how the temperature $\theta$ is transported without any thermal-mechanical processes. We define $\frac{\tilde{D}}{\tilde{D} t}$ as the mechanical time derivative, which describes the mechanical transport properties, known as the kinematics of temperature. We consider three possible situations, which lead to three distinct systems:
\begin{itemize}
    \item[(A1)] We assume that the temperature moves along with the material particles, i.e. $\theta$ is transported along the trajectory of the flow map with velocity $u$, given by
    
    $$\frac{\tilde{D}}{\tilde{D} t} \theta = \frac{\dd}{\dd t} \theta(x(X,t),t)  =  \partial_t \theta + u\cdot \nabla \theta. $$

    \item[(A2)] We assume that the temperature is on a fixed background, i.e. $\theta$ is independent of the flow map with velocity $u$. This relation is equivalent to the temperature satisfying the equation: $$\frac{\tilde{D}}{\tilde{D} t} \theta = \partial_t \theta.$$

    \item[(A3)] We assume that the temperature moves independent from the particles on the macroscopic scale, 
    but is transported along the trajectory of another flow map, the macroscopic flow map $y(X,t)$, defined by a macroscopic background velocity $v$.  For example, $v$ could represent the velocity in an incompressible Navier-Stokes equation. In this case,
     $$\frac{\tilde{D}}{\tilde{D} t} \theta = \frac{\dd}{\dd t} \theta(y(X,t),t)  =  \partial_t \theta + v \cdot \nabla \theta.$$
\end{itemize}

\subsection{EnVarA and force balance}

The key idea of the EnVarA is to derive the conservative and dissipative forces from a prescribed free energy and a mechanical dissipation functional using the least action principle and the maximum dissipation principle.

We start with the least action principle to derive the first contributions to the force balance equation. The least action principle states that the path taken by a Hamiltonian system between two states is the one for which the action functional $A$ is stationary (usually a minimum). This principle can be used to derive the the inertial and conservative forces in a dissipative system. More precisely, for a given initial volume $\Omega_0 \subseteq \Omega$ (the reference configuration), its evolution at a specific time $t \in (0, T)$ is represented by $\Omega_t^x = \{ z \in \Omega : z = x(X, t),~X \in \Omega_0 \}$, where $x(X, t)$ is the flow map introduced in (\ref{flow-map}), and $\Omega_t^x$ is referred to as the current configuration.
We can define the action $A$ as the integral of 
the kinetic energy minus the free energy over time $[0, T]$ as follows
\begin{align*}
A(x(\cdot))= \int_0^T \left( 
 E^{kin} - \int_{\Omega_t^x}\psi(\phi,\nabla\phi,\theta) \, \textnormal{d} x \right) \textnormal{d}t \ ,
\end{align*}
where $E^{kin}$ is the kinetic energy.
One shall consider the action $A$ as a functional of the flow map $x(X, t)$ in Lagrangian coordinates, as will be done later in  \eqref{action-in-lagrangian}.
The inertial and conservative forces, $f_{inertial}$ and $f_{cons}$, can be obtained by the variation of $A$ with respect to the flow map $x(X, t)$, i.e.,
\begin{equation}\label{eq:action-fcons}
\delta A(x(\cdot))=\frac{d}{d\varepsilon}\bigg|_{\varepsilon=0} A(x(\cdot) + \varepsilon \delta x(\cdot)) = 
 \int_0^T\int_{\Omega_t^x} (f_{inertial} - f_{cons}) \cdot \delta x \, \textnormal{d} x \textnormal{d}t,
\end{equation}
where $\delta x$ is any suitable smooth perturbation of the flow map $x(X, t)$.
%satisfies $\delta x(X, 0)$ and $\delta x(X, T) = 0$ and vanish. 
 
%The first result now, provides an explicit form of $f_{cons}$, 
Depending on the different kinematics (A1), (A2) or (A3) for the temperature $\theta$, the least action principle leads to different forces.
\begin{theorem}\label{thm:conservative-forces}
    Assume that the velocity $u$ in \eqref{eq: kinematics} generates an unique smooth flow map $x(X,t)$.
    \begin{enumerate}[label=(\roman*)]
        \item If the temperature $\theta$ satisfies the assumption (A1) and $E^{kin} = 0$, then
        \begin{equation*}
           f_{inertial}  = 0, \quad f_{cons} =  \mu \nabla \phi  + s  \nabla \theta.
        \end{equation*}
        \item If the temperature $\theta$ satisfies the assumption (A2) and $E^{kin} = 0$, then
        \begin{equation*}
           f_{inertial} = 0, \quad f_{cons} = \mu \nabla \phi.
        \end{equation*}
        \item  If the temperature $\theta$ satisfies the assumption (A3) and $E^{kin} = \int \frac{1}{2} |v|^2 \dd x$, where $v$ satisfies the incompressible condition, $\nabla \cdot v = 0$, then
            \begin{equation*}
            \begin{aligned}
          & f_{inertial}^u = 0, \quad f_{cons}^u=  \mu \nabla \phi, \\
          & f_{inertial}^v = - (\partial_t v + v \cdot \nabla v) \quad  f_{cons}^v= \nabla p +s\nabla \theta \\  %\nabla(\frac{1}{2} |v|^2-\psi(\theta))=: \nabla p. \\
          \end{aligned}
            \end{equation*}
    \end{enumerate}
\end{theorem}

\begin{proof}
In this proof, we focus on case $(i)$ and only present the main differences for cases $(ii)$ and $(iii)$.

Let us assume that $(A1)$ holds, where we have that $\theta(x(X,t),t) = \theta_0(X)$, with $\theta_0$ being the initial temperature. Then, in order to compute the variation of the action, we first recast the action functional in Lagrangian coordinates. This yields
\begin{equation}\label{action-in-lagrangian}
A(x(\cdot))=-\int_{t\in [0,T]}\int_{X\in \Omega_0} 
    \psi
    \big(
    \underbrace{
       \phi_0(X),F(X,t)^{-T}\nabla_X \phi_0(X),\theta_0(X)
    }_{=:\mathcal{G}(X,t)}
    \big)\det F(X,t) \,\textnormal{d}X \textnormal{d}t,
\end{equation}
where $F(X,t) = \nabla_X x(t,X) = (\partial x_i/\partial X_j)_{i,j=1,\dots,d}$ denotes the deformation tensor.
For this functional, we consider a variation of the form
\begin{align*}
    x^\varepsilon(X,t):=x(X,t)+\epsilon \delta x(X,t)
\end{align*}
 of the trajectory $x(X,t)$, where $\delta x(X,t)$ is an arbitrary smooth vector with compact support on $\mathring{\Omega}_0 \times (0,T)$.
Thus,
\begin{align*}
\delta  A(x(\cdot ))
:&=
    \frac{d}{d\epsilon}\bigg|_{\epsilon=0}A(x^\epsilon(\cdot ))\\
   &= 
    \int_{t\in [0,T]}
    \int_{X\in \Omega_0}
    \partial_{\nabla \phi}\psi \big(\mathcal{G}(X,t)\big)
    \cdot 
        F^{-T}\nabla_X \delta x F^{-T} \nabla_X \phi_0
       \det F \,\textnormal{d}X  \textnormal{d}t  \\
    &\quad - \int_{t\in [0,T]}
    \int_{X\in \Omega_0}
    \psi\big( \mathcal{G}(X,t)\big) 
         \tr\big(F^{-T}\nabla_X \delta x\big)
         \det F \,\textnormal{d}X  \textnormal{d}t.
\end{align*}
At this stage, we recast the integral back in Eulerian coordinates:
\begin{align*}
\delta  A(x(\cdot ))
&= 
\int_{t\in [0,T]}
\int_{\Omega_t^x}
    \partial_{\nabla \phi}\psi 
    \big(\phi,\nabla_x \phi ,\theta\big)\otimes \nabla_x\phi :  \nabla_x \delta x 
    -  \psi\big(\phi,\nabla_x \phi ,\theta\big)\nabla_x\cdot\delta x \,
 \textnormal{d}x  \textnormal{d}t.
\end{align*}
Using integration by parts yields
\begin{align*}
\delta  A(x(\cdot ))
&=-
\int_{t\in [0,T]}
\int_{\Omega_t^x} \big\{
    \nabla_x \cdot \big(\partial_{\nabla \phi}\psi \otimes \nabla \phi\big)  - \nabla_x \psi  \big\}\cdot\delta x \,\textnormal{d}x  \textnormal{d}t,
\end{align*}
where the boundary terms on $\partial \Omega^x_t$ vanish because $\delta x$ is supported in the interior $\mathring{\Omega}_0\times (0,T)$. The function in front of $\delta x$ can be further simplified as follows: 
\begin{align*}
\nabla\cdot\big(\partial_{\nabla \phi} \psi \otimes \nabla \phi\big)-\nabla \psi &=\nabla\cdot\partial_{\nabla \phi} \psi \nabla \phi+\partial_{\nabla \phi} \psi \nabla^2\phi -\partial_\phi \psi\nabla \phi-\partial_{\nabla \phi} \psi \nabla^2\phi -\partial_\theta \psi \nabla \theta \\
& = \big(\nabla\cdot\partial_{\nabla \phi} \psi -\partial_\phi \psi\big) \nabla \phi-\partial_\theta \psi \nabla \theta = \mu \nabla \phi +s \nabla\theta,
\end{align*}
where $\mu=\mu(\phi,\theta)=\nabla\cdot\partial_{\nabla \phi} \psi -\partial_\phi \psi$ denotes the chemical potential.
Thus, we conclude the first part of the proof by recalling identity \eqref{eq:action-fcons}.

Assuming that $(A2)$ holds, the proof follows along the same lines, with an additional term $\partial_\theta \psi \nabla \theta$ appearing during the last step of the least action principle.
%where in the last step of the least action principle an additional term $\partial_\theta \psi \nabla \theta$ would appear in the last integral.
Therefore, in this case, the least action principle yields the conservative force
\begin{equation*}
    f_{cons}= \mu \nabla \phi.
\end{equation*}

Now, let assumption (A3) for the temperature hold. Since there are two flow maps $x(X, t)$ and $y(X, t)$ associated with the velocity $u$ and $v$ respectively, we have to apply the least action principle to both flow maps.
For the flow map $x(X,t)$ moving with velocity $u$, we obtain the same contribution to the conservative forces as in case (A2) as $\theta$ is not transported by $u$.
However, as the temperature is transported by $v$ we obtain for the flow map $y(X,t)$ moving with velocity $v$ that 
\begin{align*}
  \delta  A(y(\cdot )) &= 
    % \int_{t\in [0,T]}     \int_{\Omega_t^y}  \big\{ -(\partial_t v+v\cdot\nabla v)-\nabla\big(\frac{1}{2}|v|^2-\psi(\theta)\big)\big\}\cdot \delta y \,\textnormal{d}y\textnormal{d} t,
    \int_{t\in [0,T]}     \int_{\Omega_t^y}  \{ -(\partial_t v+v\cdot\nabla v)-\nabla p + s \nabla \theta\} \cdot \delta y \,\textnormal{d}y\textnormal{d} t,
\end{align*}
where $p$ is the Lagrange multiplier for the incompressible condition, $\det F = 1$.
The last term follows from the fact that the free energy under the flow $y$ only depends on the temperature $\theta$.
\end{proof}

The next step in the EnVarA is to compute the dissipative forces of the system using the maximum dissipation principle, which states that these forces  can be obtained by the variation
of the so-called dissipation functional $\mathcal{D}$ with respect to the velocity fields. 
In the current study, we consider a dissipation functional of the following form
\begin{align}\label{eq:dissipation}
    \mathcal{D}=\int_\Omega  \eta(\phi,\theta)|(u - v)\cdot\nabla \phi|^2 +\nu |\nabla v|^2\, \textnormal{d}x ,
\end{align}
where $v = 0$ in cases (A1) and (A2). The difference $u-v$ is the so-called effective velocity of the system.
Here, $\eta$ denotes a positive friction parameter, depending on the state variables $\phi$ and $\theta$. In the setting of phase transitions the inverse of this parameter is also known as the mobility. In the classical Allen-Cahn equation, $\eta$ is typically chosen to be $O(\epsilon)$, in order to see the evolution of the interface \cite{bronsard1991motion, ilmanen1993convergence}.

\begin{remark}
The choice of dissipation (\ref{eq:dissipation}) is consistent with that in the classical Allen-Cahn equation, which is often interpreted as an $L^2$-gradient flow associated with a given free energy $\mathcal{F}(\phi) = \int \psi \dd x$. The dissipation of a Allen-Cahn equation with convection velocity $v$, is often written as \cite{giga2017variational}
\begin{equation}
 \mathcal{D} = \int_{\Omega} \eta |\phi_t + v \cdot \nabla \phi|^2 \dd \x\ ,
\end{equation}
where $v$ is the macroscopic background velocity.
Due to the kinematics (\ref{eq: kinematics}), one can replace $\phi_t$ by $- u \cdot \nabla \phi$ \cite{liu2020variational}, which leads to the first term in (\ref{eq:dissipation}).
\end{remark}

Mathematically, the maximum dissipation principle states that
\begin{align*}
    \frac{1}{2}\delta \mathcal{D}(u)= \int_\Omega f_{diss} \cdot \delta u \,\textnormal{d}u
\end{align*}
and similar for the dissipation force associated with the velocity $v$.
Hence, for cases (A1) and (A2), we have
\begin{align*}
    f_{diss}= \eta(\phi,\theta) (u \cdot \nabla \phi) \nabla \phi
\end{align*}
and for case (A3), we have
\begin{align*}
    f_{diss}^u&= \eta(\phi,\theta) ( (u - v) \cdot \nabla \phi ) \nabla \phi,\\
    f_{diss}^v&= \eta(\phi,\theta) ( (v - u) \cdot \nabla \phi ) \nabla \phi -\nu \Delta v.
\end{align*}
Applying Newton's force balance law we obtain the following result:
\begin{theorem}\label{thm: force balance}
  Assume that the assumptions on the flow maps of the previous theorem still hold. 
  Then
  \begin{enumerate}[label=(\roman*)]
      \item If the temperature satisfies the assumption (A1), then the equation of motion for $\phi$ is given by
      \begin{align*}
          \partial_t\phi +u\cdot\nabla\phi = 0, \quad  - \eta(\phi,\theta) (u \cdot \nabla \phi) \nabla \phi  = \mu \nabla \phi +s\nabla\theta.
      \end{align*}
      \item If the temperature satisfies the assumption (A2), then the equation of motion for $\phi$ is given by 
      \begin{align*}
          \partial_t\phi +u\cdot\nabla\phi = 0, \quad - \eta(\phi,\theta) (u \cdot \nabla \phi) \nabla \phi  = \mu \nabla \phi.
      \end{align*}
     which can be rewritten as
      \begin{align}\label{non_ac_1}
           \partial_t\phi = - \eta(\phi,\theta)^{-1} \mu  % \big(\nabla\cdot\partial_{\nabla \phi} \psi -\partial_\phi \psi\big).
      \end{align}
The equation (\ref{non_ac_1}) is the classical Allen-Cahn equation when $\theta = \theta_c$.

      \item If the temperature satisfies the assumption (A3), then the equation of motion for $\phi$ is given by
      \begin{align*}
           \partial_t\phi +u\cdot\nabla\phi =0, \quad - \eta(\phi,\theta) ((u-v) \cdot \nabla \phi ) \nabla \phi  = \mu \nabla \phi
      \end{align*}
      and the equation  for $v$ is given by
      \begin{align*}
          \partial_t v+ (v\cdot\nabla) v +\nabla p&= -\mu \nabla \phi + \nu \Delta v -s\nabla \theta\\
          \nabla\cdot v&=0.
      \end{align*}
  \end{enumerate}
\end{theorem}
\begin{proof}
    %With both conservative and dissipative forces derived 
    The force balance condition states in general
    \begin{align*}
        f_{cons}+ f_{diss}=f_{inertial}.
    \end{align*}
    For cases (A1) and (A2) there are no inertial forces and thus the right-hand side equals zero.
    For case (A3) we have to balance the forces for each flow map and keep in mind the inertial forces arriving for the $v$-equation.
\end{proof}

\subsection{The temperature equation}

What remains now is to derive the equation of the temperature $\theta$ or entropy $s$. The derivation relies on the first and second laws of thermodynamics, with the latter expressed through the Clausius-Duhem inequality.

The Clausius-Duhem inequality provides a mathematical expression of the second law of thermodynamics, the increase of the total entropy.  It can be written in the following equivalent forms:
\begin{equation}\label{eq:clausius-duhem}
\begin{aligned}
  &\text{(A1)} \qquad  \theta ( \partial_t s +\nabla \cdot (us) )  = \theta \nabla \cdot j + \theta \Delta^*, \\
  &\text{(A2)} \qquad  \theta \partial_t s = \theta \nabla \cdot j + \theta \Delta^*,\\
  &\text{(A3)}\qquad \theta( \partial_t s +\nabla \cdot (vs)) = \theta \nabla \cdot j + \theta \Delta^*,
\end{aligned}
\end{equation}
where $j$ is the entropy flux specified by
\begin{align*}
    j=- \frac{q}{\theta},
\end{align*}
with $q$ being the heat flux, and $\Delta^*\geq 0$ is the point-wise rate of entropy production.
The different forms of the entropy equation again come from the different assumptions on the kinematics of the temperature $\theta$ and the duality of the entropy $s$ with respect to the temperature $\theta$. 

According to the Fourier's law, the heat flux is given by
\begin{equation}\label{Fourier}
    q=k(\phi,\theta) \nabla\theta,
\end{equation}
where $k=k(\phi,\theta)>0$ is the thermal conductivity.
Hence, the equation of the entropy $s$ is determined as if the the rate of entropy production $\Delta^*$ is given.

\begin{remark}
   The Fourier law is a classical assumption on the heat flux. Alternative formulations have also been proposed in the literature. For instance, a rescaled version, $q=k\frac{\nabla \theta}{\theta}$, can be found in the literature, see for example \cite{miranville2005nonisothermal}.
   Another choice is to assume a time dependent relation $\partial_t q + q=k\nabla\theta$, which is known as Cattaneo's law \cite{mariano2022sources}. 
\end{remark}

To determine the form of $\Delta^*$, one need to use first law of thermodynamics, which states that the total energy for an isolated system, where no work is done on or by the system and no heat is exchanged with the environment, remains constant, i.e., $\frac{\textnormal{d}}{\textnormal{d}t} E^{tot}(t)=0$. To enforce this condition, we assume that the heat flux satisfies a non-flux boundary condition, i.e.,
\begin{align}\label{eq: heat flux boundary}
    q\cdot n=0.
\end{align}
To define the total energy we introduce the internal energy density $e = e(\phi,\theta)$ given by a Legendre transform of the free energy with respect to the temperature
\begin{align}\label{eq: internal energy}
    e(\phi,\theta):=\psi(\phi,\theta)-\theta\partial_\theta \psi(\phi,\theta) = \psi(\phi,\theta)+ \theta s.
\end{align}
Then, the total energy is given by
\begin{align}\label{eq: total energy}
    E^{tot}(t)= \int_\Omega e(\phi,\theta) \,\textnormal{d}x + E^{kin}(t),
\end{align}
where the last term represents the contribution from the kinetic energy, which plays a role for case (A3).
More precisely, we can have the following result:

\begin{theorem}\label{thm:entropy-production}
    Under the Fourier law (\ref{Fourier}), %if the following relations on the entropy production holds true:
    \begin{enumerate}[label=(\roman*)]
        \item If the temperature satisfies (A1) and the rate of entropy production is given by   \begin{equation}\label{eq:entropyproductionA1}
            \theta \Delta^* =   |u \cdot\nabla \phi|^2\eta(\phi,\theta) + \frac{k(\phi,\theta) |\nabla \theta|^2}{\theta},
        \end{equation}
       
        \item if the temperature satisfies (A2)  and the the rate of entropy production is given by \begin{equation}\label{eq:entropyproductionA2}
                \theta \Delta^*= |u \cdot\nabla \phi|^2\eta(\phi,\theta) + \frac{k(\phi,\theta)|\nabla\theta|^2}{\theta},
        \end{equation}

        \item If the temperature satisfies (A3) and the rate of entropy production is given by  \begin{equation}\label{eq:entropyproductionA3}
            \theta \Delta^*= |(u-v)\cdot\nabla \phi|^2\eta(\phi,\theta)+\nu |\nabla v|^2 + \frac{k(\phi,\theta)|\nabla\theta|^2}{\theta},
        \end{equation}

    \end{enumerate}
     then the total energy of the system is conserved.
\end{theorem}

\begin{proof}
    We start by computing the the following
    \begin{align*}
        \frac{\textnormal{d}}{\textnormal{d}t}E^{tot}(t)= \frac{\textnormal{d}}{\textnormal{d}t} \int_\Omega e(\phi,\theta)\, \textnormal{d}x+ \frac{\textnormal{d}}{\textnormal{d}t}E^{kin}(t)
    \end{align*}
For the first term we have 
\begin{align*}
    & \frac{\textnormal{d}}{\textnormal{d}t} \int_\Omega e(\phi,\theta)\, \textnormal{d} x= \int_\Omega  \partial_t [\psi(\phi, \theta)] + \partial_t[\theta\, s(\phi, \theta)] \,\textnormal{d} x\\
    & =\int_\Omega \partial_\phi \psi(\phi, \theta) \partial_t \phi + \partial_{\nabla \phi}\psi(\phi,\theta) \cdot \partial_t \nabla \phi + \partial_\theta \psi(\phi, \theta) \partial_t \theta + (\partial_t \theta)s(\phi, \theta) + \theta \partial_t s(\phi, \theta) \,\textnormal{d} x. 
\end{align*}

Recalling that $ \partial_\theta \psi(\phi, \theta)=-s(\phi,\theta)$ and using integration by parts, we obtain
    \begin{align*}
         \frac{\textnormal{d}}{\textnormal{d}t} \int_\Omega e(\phi,\theta)\, \textnormal{d} x&= \int_\Omega ( \psi(\phi, \theta)- \nabla\cdot\partial_{\nabla \phi}\psi(\phi,\theta))\partial_t\phi + \theta\partial_t [s(\phi, \theta)] \,\textnormal{d} x\\
        &=\int_\Omega - ( \psi(\phi, \theta)- \nabla\cdot\partial_{\nabla \phi}\psi(\phi,\theta))u\cdot\nabla\phi + \theta\partial_t s(\phi, \theta) \,\textnormal{d} x\\
         &=\int_\Omega \mu(\phi,\theta)u\cdot\nabla\phi + \theta\partial_t s(\phi, \theta) \,\textnormal{d} x,
    \end{align*}
where we used the definition of the chemical potential $\mu=\mu(\phi,\theta)$.
To proceed, we have to use the different assumptions on the temperature. 

If (A1) holds, then we have by relation \eqref{eq: internal energy} and \eqref{eq: total energy} that
\begin{align*}
     \frac{\textnormal{d}}{\textnormal{d}t} \int_\Omega e(\phi,\theta)\, \textnormal{d} x&= \int_\Omega u\cdot(-u\cdot \nabla \phi\nabla\phi \eta(\phi,\theta)-s\nabla\theta) -\theta \nabla\cdot(s u) +\theta\nabla\cdot j+\theta\Delta^*\textnormal{d} x\\
     &=\int_\Omega -|u\cdot\nabla\phi|^2 \eta(\phi,\theta) +\nabla\cdot q-\frac{k(\phi,\theta)|\nabla\theta|^2}{\theta}+\theta\Delta^*\, \textnormal{d} x. 
\end{align*}
Using the boundary conditions for the heat flux $q$ \eqref{eq: heat flux boundary}, we have
\begin{align*}
    \frac{\textnormal{d}}{\textnormal{d}t}E^{tot}(t)= \int_\Omega -|u\cdot\nabla\phi|^2 \eta(\phi,\theta) +\nabla\cdot q-\frac{k(\phi,\theta)|\nabla\theta|^2}{\theta}+\theta\Delta^* \, \textnormal{d} x=0
\end{align*}
as long as identity \eqref{eq:entropyproductionA1} is satisfied.
This concludes the proof Theorem \ref{thm:entropy-production}, part $(i)$. The argument to prove $(ii)$ under assumption (A2) is analogous.

Now, let us assume that assumption (A3) holds. In this case, we must consider the contribution of the kinetic energy to the total energy. Hence,
    \begin{align*}
         \frac{\textnormal{d}}{\textnormal{d}t}E^{kin}(t)& = \frac{d}{dt}\int_\Omega \frac{1}{2}|v|^2\, \textnormal{d}x
        = \int_\Omega (\partial_t v+ (v\cdot\nabla)v)\cdot v  \,\textnormal{d}x\\
         &= \int_\Omega ((u-v)\cdot\nabla\phi\nabla\phi \eta(\phi,\theta)+ \nu \Delta v-\nabla p-s\nabla \theta)\cdot v \, \textnormal{d}x.
    \end{align*}
    From the change in the internal energy we have
    \begin{align*}
          \frac{\textnormal{d}}{\textnormal{d}t} \int_\Omega e(\phi,\theta)\, \textnormal{d} x&= \int_\Omega u\cdot (-(u-v)\cdot\nabla\phi\nabla\phi \eta(\phi,\theta))  -\theta \nabla\cdot(s v) +\theta\nabla\cdot j+\theta\Delta^* \,\textnormal{d} x.
    \end{align*}
    Combining the two integral identities and using integration by parts,  we obtain
    \begin{align*}
         \frac{\textnormal{d}}{\textnormal{d}t}E^{tot}(t)&=  \int_\Omega -|(u-v)\cdot\nabla \phi|^2\eta(\phi,\theta)-\nu |\nabla v|^2 -\frac{k(\phi,\theta)|\nabla\theta|^2}{\theta}+\theta\Delta^* \, \textnormal{d} x\\
        &+\int_\Omega -v\cdot\nabla p +\nabla\cdot q \, \textnormal{d} x=0,
    \end{align*}
where the last integral vanishes due to the incompressibility of $v$, the non-flux boundary condition for the heat flux (\ref{eq: heat flux boundary}), and the boundary condition $v \cdot n = 0$.
Thus, the identity holds as long as identity \eqref{eq:entropyproductionA3} is satisfied. 
This concludes the proof of the theorem.   
\end{proof}

\subsection{Summary of non-isothermal systems}
Now, we can summarize three non-isothermal diffuse interface models derived in this section by combining the results from Theorems \ref{thm:conservative-forces}, \ref{thm: force balance} and \ref{thm:entropy-production}. The equations are defined on the domain $(x,t)\in \Omega\times (0,T)$ and they shall be supplemented by a proper boundary conditions on  $\partial \Omega\times (0,T)$, say the non-flux boundary condition, given by $\nabla\phi\cdot n=0,~~\nabla \theta\cdot n =0,~~ u\cdot n=0$. 
    
Recall the definitions of the entropy $s$ and the chemical potential $\mu$
\begin{align*}
    s(\phi,\theta)&= \chi(\epsilon)\frac{\lambda}{\theta_c}g(\phi)+ c_v \log \theta,\\
    \mu(\phi,\theta)&= \epsilon\Delta \phi-\frac{1}{\epsilon} W'(\phi)-\chi(\epsilon)\frac{\lambda}{\theta_c}(\theta-\theta_c)g'(\phi).
\end{align*}
Assuming that the assumption (A1) of Section \ref{sec: kinematics of temp} holds true, the following system on $(\phi(x,t),\theta(x,t))$ satisfies both the first and second laws of thermodynamics:
\begin{equation*}
\begin{cases}
     & \partial_t \phi +u \cdot \nabla\phi =0, \quad -u\cdot \nabla \phi \nabla \phi \eta(\phi,\theta) = \mu \nabla \phi +s\nabla\theta,\\
     & \theta (\partial_t s+ \nabla\cdot (us)) = \theta \nabla \cdot \left( \frac{\kappa(\phi,\theta)\nabla\theta}{\theta}\right) + \theta \Delta^*, \quad \theta\Delta^* = |u\nabla \phi|^2\eta(\phi,\theta)+ \frac{\kappa(\phi,\theta)|\nabla\theta|^2}{\theta},\\
\end{cases}
\end{equation*}

\begin{remark}
We note that at least formally we can combine the first two equations yielding
   \begin{align*}
       \partial_t\phi= \eta(\phi,\theta)^{-1} \mu + \frac{s \nabla\theta \cdot \nabla \phi}{\eta(\phi,\theta) |\nabla\phi|^2}.
   \end{align*}
   Thus, states with constant phase field value are penalized unless the temperature is also constant.
   Therefore, it is a reasonable choice to work with the logarithmic potential in this setting as it prevents the pure states where $\phi=\pm 1$.
   Additionally, we observe that due to assumption (A1) and the resulting coupling between temperature and phase variable, the temperature also evolves on a fast scale.
\end{remark}

Let assumption (A2) hold. Then, the following system on the state variables $(\phi(x,t),\theta(x,t))$ satisfies both the first and second laws of thermodynamics
\begin{equation*}
\begin{cases}
     &  \partial_t \phi +u \cdot \nabla\phi =0, \quad -u\cdot \nabla \phi \nabla \phi \eta(\phi,\theta) = \mu \nabla \phi,\\
    & \theta \partial_t s = \theta \nabla \cdot \left( \frac{\kappa(\phi,\theta)\nabla\theta}{\theta}\right) + \theta \Delta^*, \quad \theta\Delta^* = |u\nabla \phi|^2\eta(\phi,\theta)+ \frac{\kappa(\phi,\theta)|\nabla\theta|^2}{\theta},\\
    % \nabla\phi\cdot n=0,\quad \nabla \theta\cdot &=0,\quad u\cdot n=0,\\
    % \phi(0)=\phi_0,\quad &\theta(0)=\theta_0.
\end{cases}
\end{equation*}
In this setting, we can combine the first two equations, which give the classical Allen-Cahn equation and in addition we explicitly compute the terms in the entropy equation. Hence, we obtain
  \begin{align*}
      \partial_t \phi& =  - \eta(\phi,\theta)^{-1} \mu,\\
      c_v \partial_t \theta&= \nabla\cdot (\kappa(\phi,\theta)\nabla\theta)-\chi(\epsilon)\frac{\lambda}{\theta_c}\theta g'(\phi)\partial_t \phi +\eta(\phi,\theta)|\partial_t \phi|^2.
  \end{align*}
  For a system of this form the authors of \cite{lasarzik2022weak} were able to show the well-posedness of weak solutions.

Finally, let assumption (A3) hold.  Then, the following system on the state variables $(\phi(x,t),\theta(x,t))$ satisfies both the first and second laws of thermodynamics
\begin{equation*}
\begin{cases}
    & \partial_t\phi +u\cdot\nabla\phi = 0, \quad -(u-v) \cdot\nabla\phi \nabla \phi \eta(\phi,\theta) = \mu \nabla \phi,\\
     & \partial_t v+ (v\cdot\nabla)v +\nabla p = \lambda\Delta v -\mu \nabla\phi -s\nabla\theta, \quad \nabla \cdot v =0,\\
    & \theta( \partial_t s+ \nabla\cdot (sv)) = \theta \nabla \cdot \left( \frac{\kappa(\phi,\theta)\nabla\theta}{\theta}\right) + \theta \Delta^*,  \\
    &   \theta\Delta^* = |(u-v)\nabla \phi|^2\eta(\phi,\theta)+ \lambda|\nabla v|^2 + \frac{\kappa(\phi,\theta)|\nabla\theta|^2}{\theta},\\
    % \nabla\phi\cdot n=0,\quad \nabla \theta\cdot &=0,\quad u\cdot n=0,\\
    % \phi(0)=\phi_0,\quad &\theta(0)=\theta_0.
\end{cases}
\end{equation*}
Simplify the above equations, we end up with a coupled Allen-Cahn-Navier-Stokes-Fourier system
    \begin{align*}
        & \partial_t \phi = \eta(\phi,\theta)^{-1} \mu  - v\cdot\nabla \phi,\\
        & \partial_t v+ (v\cdot\nabla)v +\nabla p = \lambda\Delta v -\mu \nabla\phi -(\chi(\epsilon)\frac{\lambda}{\theta_c}g(\phi)+ c_v \log \theta)\nabla\theta, \quad \nabla \cdot v =0,\\
     & c_v \partial_t \theta+ c_v v\cdot\nabla	\theta  =\nabla\cdot (\kappa(\phi,\theta)\nabla\theta) -\chi(\epsilon)\frac{\lambda}{\theta_c}\theta g'(\phi) \frac{\mu}{\eta(\phi,\theta)}+  \frac{|\mu|^2}{\eta(\phi,\theta)}. 
    \end{align*}

\begin{remark}
%In some recent work \cite{}, the authors propose a free energy takes the form of
% Following the idea in recent work on non-isothermal phase-field crystal models \cite{sun2024thermodynamically}, 
Since the starting point of model derivation is the form of the free energy, different choices of free energy can lead to different models. For example, one might choose the free energy as
\begin{equation}
\psi(\phi,\theta)  = \theta( \frac{\epsilon}{2} |\nabla \phi|^2 + \frac{1}{\epsilon} W(\phi)  + \frac{\chi}{2 \sigma} (\theta - \theta_c) \phi ) + c_v \theta (1 - \log \theta),
\end{equation}
as used in a recent work on non-isothermal phase-field crystal models \cite{sun2024thermodynamically}. Then the entropy is given by
\begin{equation}
s = - \psi_{\theta} = - \frac{\chi}{2 \sigma} (2 \theta - \theta_c) \phi + c_v \log \theta
\end{equation}
and the internal energy is given by
\begin{equation}
  e = \psi + \theta s = c_v \theta - \frac{\chi}{2 \sigma} \theta_c \phi.
\end{equation}
% With this free energy we can recover 
This choice of free energy leads to a heat equation that similar to that in Caginalp's model \cite{caginalp1986analysis,caginalp1990dynamics,caginalp1995derivation,caginalp1998convergence}, in which the temperature equation is linearly depend on $\theta$ and $\phi$.
\end{remark}

\section{The Stefan type problem as a reduced model of the non-isothermal Allen-Cahn equation}\label{sect: Limit}
In this section, we show that the sharp interface limit of the non-isothermal Allen-Cahn model (model with the kinematics assumption (A2) on the temperature) leads to a two-phase Stefan type problem.
We utilize the method of matched asymptotics, which has been successfully applied to various phase field models \cite{blowey1993curvature,garcke2014diffuse,garcke2016cahn}. The idea behind a matched asymptotic expansion is to construct an approximation in the form of an asymptotic series that is obtained by splitting the domain into two parts; the inner problem and the outer problem, named for their relationship to the transition layer; and treating each part of the domain as a separate perturbation problem. The outer and inner solutions are then combined through a matching process such that an approximate solution for the whole domain is obtained.

We assume that $W(\phi)$ is the standard double-well potential and $g(\phi)= -\frac{1}{3}\phi^3+\phi$. %Additionally, we consider 
In order to obtain a Stefan type problem as the sharp interface limit, it is crucial to use the following scaling of the parameters:
$\chi(\epsilon)= 1$ and $\eta(\phi,\theta)= \epsilon$.
The final non-isothermal Allen-Cahn model considered is
\begin{align} \label{phi_equation}
    \partial_t \phi &= \Delta \phi -\frac{1}{\varepsilon^2}\phi(\phi^2-1) -\frac{1}{\varepsilon}\frac{\lambda}{\theta_c}(\theta-\theta_c)(\phi^2-1),\\ \label{temp_equation}
    c_v\partial_t \theta&= \nabla\cdot \big(k(\phi)\nabla \theta\big)+\frac{\lambda}{\theta_c}(\phi^2-1)\theta\partial_t \phi + \epsilon|\partial_t\phi|^2,
\end{align}
in a bounded domain $\Omega\subset \mathbb{R}^d$ supplemented with the non-flux boundary conditions
\begin{align}\label{boundary_condition}
    &\nabla \phi\cdot \nu=0,\quad \nabla \theta\cdot \nu=0
\end{align}
and the initial conditions
\begin{align}\label{initial_condition}
  &\phi(0)=\phi^0,\quad \theta(0)=\theta^0.  
\end{align}
% The goal is 

We assume that there is a family of solutions $(\phi_\varepsilon,\theta_\varepsilon)_{\varepsilon>0}$ which are sufficiently smooth and have an asymptotic expansion in $\varepsilon$ in the bulk regions, i.e. the regions of pure phase, away from the interface, and another expansion in the interfacial region. In addition, there are certain matching conditions that relate the two asymptotic expansions.
Moreover, we assume that the zero level sets of $\phi_\varepsilon$ converge to a limiting hypersurface $\Gamma_0$ moving with normal velocity $\mathcal{V}$.

\begin{remark}
    For simplicity we assume that the interface region does not touch the outer boundary $\partial \Omega$.
\end{remark}

\subsection{Outer expansion}
We assume that for $f_\varepsilon \in \{\phi_\varepsilon,\theta_\varepsilon \}$ the following outer expansion holds
\begin{align*}
    f_\varepsilon= f_0+\varepsilon f_1 +\varepsilon^2 f_2+\mathcal{O}(\varepsilon^3).
\end{align*}
Substituting this expansion into the phase-field equation \eqref{phi_equation}, the leading-order term at $\mathcal{O}(\varepsilon^{-2})$ yields
\begin{align}\label{eq: outer expansion phi}
    W'(\phi_0)=\phi_0(\phi_0^2-1)=0.
\end{align}
The solutions to \eqref{eq: outer expansion phi} are $\phi=0$ and $\phi = \pm 1$. Among these, $\phi = \pm 1$ correspond to the minima of the potential $W$, while $\phi = 0$ is a saddle point.
Thus, in the sharp-interface limit, the domain separates into two bulk regions corresponding to the pure phases:
\begin{align}
    \Omega_+:=\{x\in \Omega:\, \phi_0(x)=+1\},\quad \Omega_-:=\{x\in \Omega:\, \phi_0(x)=-1\}
\end{align}
and the interface separating them is defined by
\begin{align}
    \Gamma_0:=\{x\in \Omega:\, \phi_0(x)=0\}.
\end{align}
At the next order, $\mathcal{O}(\varepsilon^{-1})$, equation \eqref{phi_equation} gives the interfacial condition:
% that on the interface $\Gamma_0$ we have that $\theta_0=\theta_c$. 
\begin{align}\label{eq: interface theta}
    \theta_0(x)=\theta_c ~~\text{for } x \in \Gamma_0\ ,
\end{align}
which indicates that the temperature must match the critical value $\theta_c$ on the interface.

Now, consider the temperature equation \eqref{temp_equation}. In the pure phase regions $\Omega_-$ and $\Omega_+$, we have $\partial_t \phi_0=0$. Moreover, by the smoothness of $k(\phi)$, we define the constant thermal conductivities in each phase as 
$$  k_{+1} = k(1), \quad k_{-1}  = k(-1) $$
%$k_1=k(\phi=1)$ in $\Omega_-$ and $k_{-1}=k(\phi= -1)$ in $\Omega_+$, respectively.
Then, at leading order $\mathcal{O}(1)$ in \eqref{temp_equation}, we obtain the classical heat equation in each region:
\begin{align}
   c_v \partial_t \theta_0 = k_{\pm 1} \Delta \theta_0, \quad \textnormal{in } \Omega_\pm.
\end{align}

\subsection{Inner expansion and matching conditions}
We now perform an inner expansion in the interfacial region, where the transition between the two phases takes place.

We denote $\Gamma_0(t)$ the limiting hypersurface defined by the zero level sets of $\phi_{\varepsilon}(x, t)$, i.e., $\Gamma_0(t) = \{ x \in \Omega ~|~ \varphi_{\epsilon}(x, t) = 0  \}$.
To analyze the behavior near the interface, we introduce a new coordinate system. Specifically, let $d(x, t)$ be the signed distance function to $\Gamma_0(t)$ defined as $d(x, t) > 0$ in $\Omega_+$ and $d(x, t)<0$ in $\Omega_-$.  We then define the new interface variable by $z = \frac{d(x, t)}{\varepsilon}$.
Additionally, let $g(t,s)$ be a smooth parametrization of the interface $\Gamma_0(t)$ by the arc-length $s$, and let $\nu(t, s) = \nu (g(t, s))$ denote the unit normal of $\Gamma_0$ at the point $g(t, s)$ pointing into the domain $\Omega_-$.
Then, in a tubular neighborhood of $\Gamma_0(t)$, any sufficiently smooth function $f(x, t)$ can be expressed in the local $(t, s, z)$-coordinate system as
\begin{align*}
    f(x, t)= f(g(t,s)+\varepsilon z\nu(t,s), t)=:F(t,s,z).
\end{align*}
In this new $(t,s,z)$-coordinate system, the following rules of differentiation hold:
\begin{align*}
    \partial_t f(x, t)  &= - \frac{1}{\varepsilon} \mathcal{V} \partial_z F + \mathcal{O}(1),\\
    \nabla_x f (x, t) &= \frac{1}{\varepsilon}\partial_z F \cdot \nu +\nabla_{\Gamma_0} F + \mathcal{O}(\epsilon),\\
    \Delta_x f (x, t)  &= \frac{1}{\varepsilon^2} \partial_{zz} F - \frac{1}{\varepsilon}\kappa \partial_z F+ \mathcal{O}(1),
\end{align*}
where $\mathcal{V}$ is the normal velocity of the interface $\Gamma_0(t)$, $\nabla_{\Gamma_0}$ denotes the gradient on $\Gamma_0$ and $\kappa$ is the mean curvature of $\Gamma_0$.
% Here, h.o.t. denotes the higher order terms with respect to $\varepsilon$.

% Next, we denote the functions $\phi_\varepsilon,\theta_\varepsilon$ in the new coordinate system by $\Phi_\varepsilon,\Theta_\varepsilon$, respectively.
Next, we express the functions $\phi_\varepsilon$ and $\theta_\varepsilon$ in the local $(t, s, z)$-coordinate system by defining
\[
\Phi_\varepsilon(t, s, z) := \phi_\varepsilon(x, t), \quad \Theta_\varepsilon(t, s, z) := \theta_\varepsilon(x, t),
\]
where $x = g(t, s) + \varepsilon z\, \nu(t, s)$.
Then, we further assume that for $F_\varepsilon\in \{\Phi_\varepsilon,\Theta_\varepsilon\}$,  the following inner expansion holds:
\begin{align*}
    F_\varepsilon(t, s,z) = F_0(t, s, z)+\varepsilon F_1(t, s, z)+\mathcal{O}(\varepsilon^2).
\end{align*}
Since $\Gamma_0(t)$ is the limiting zero level set of $\phi_\varepsilon$, it follows that the leading-order of the phase field function satisfies
\begin{align}\label{eq: interface phi condition}
    \Phi_0(t,s,z=0)=0.
\end{align}
Moreover, we impose the far-field conditions:
\begin{align}
    \Phi_\varepsilon(t,s,z=\infty)=1,\quad \Phi_\varepsilon(t,s,z=-\infty)=-1.
\end{align}
In addition, since the interface $\Gamma_0$ corresponds to the level set of the critical temperature $\theta_c$, we have 
\begin{align}
    \Theta_0(t,s,z=0)=\theta_c.
\end{align}

In order to match the inner and outer expansions, we impose the following matching conditions:
\begin{align}\label{eq: matching 1}
    & \lim_{z\to \pm\infty} F_0(t,s,z)= f_0^\pm(t,x),\\ \label{eq: matching 2}
    & \lim_{z\to \pm \infty} \partial_z F_0(t,s,z) =0,\\ \label{eq: matching 3}
    & \lim_{z\to \pm\infty}\partial_z F_1(t,s,z) = \nabla f_0^\pm(t,x)\cdot\nu,
\end{align}
where
\begin{align*}
    f_0^\pm(t,x):=\lim_{\delta\to 0} f_0(t,x\pm \delta\nu)
\end{align*}
for $x\in \Gamma_0$.
Furthermore, let $\delta>0$ and $x\in \Gamma_0$ with $x-\delta\nu \in \Omega_+$ and $x+\delta\nu\in \Omega_-$. Then, we denote the jump of a quantity $f$ across the interface by
\begin{align}
    [f]_+^-:=\lim_{\delta \searrow 0}f(t,x+\delta\nu)-\lim_{\delta \searrow 0}f(t,x-\delta\nu).
\end{align}
% Now, we have all the necessary tools at hand to study the inner expansion.

We now derive the leading-order equation from the phase-field model. At leading order, the terms of order $\mathcal{O}(\varepsilon^{-2})$ in the equation \eqref{phi_equation} yield the following equation
\begin{align}\label{eq: leading order phi inner}
    \partial_{zz}\Phi_0-W'(\Phi_0)=0.
\end{align}
Using \eqref{eq: interface phi condition}, we note that $\Phi_0$ can be chosen to be independent of $s$ and $t$, i.e. $\Phi_0$ is only a function of $z$.
Thus, $\Phi_0$ solves the ODE
\begin{align}
    \Phi_0''(z)-W'(\Phi_0(z))=0, \quad \Phi_0(0)=0,\quad \Phi_0(\pm\infty)=\pm 1.
\end{align}
Then, for the double-well potential we have the unique solution
\begin{align}
    \Phi_0(z)= \tanh\bigg(\frac{z}{\sqrt{2}}\bigg).
\end{align}
% Additionally, we can multiply \eqref{eq: leading order phi inner} by $\Phi_0'(z)$ and obtain after integrating and applying the matching conditions the so-called equipartition of energy

To derive the energy identity, we multiply both sides of \eqref{eq: leading order phi inner}  by $\Phi_0'(z)$:
\[
\Phi_0''(z) \Phi_0'(z) = W'(\Phi_0(z)) \Phi_0'(z).
\]
Recognizing both sides as total derivatives, we integrate with respect to $z$:
\[
\frac{1}{2} \frac{d}{dz} |\Phi_0'(z)|^2 = \frac{d}{dz} W(\Phi_0(z)) \quad \Rightarrow \quad \frac{1}{2} |\Phi_0'(z)|^2 = W(\Phi_0(z)) + C,
\]
where $C$ is an integration constant. To determine $C$, we use the matching conditions: as $z \to \pm \infty$, we have $\Phi_0(z) \to \pm 1$, hence $W(\Phi_0(z)) \to 0$, and $\Phi_0'(z) \to 0$. Substituting into the identity gives
\[
\frac{1}{2} |\Phi_0'(\pm\infty)|^2 = W(\Phi_0(\pm\infty)) + C \quad \Rightarrow \quad 0 = 0 + C,
\]
so $C = 0$. Therefore, we arrive at the \textbf{pointwise equipartition of energy}:
\begin{equation}\label{eq: equipartition of energy}
\frac{1}{2} |\Phi_0'(z)|^2 = W(\Phi_0(z)) \quad \text{for all } z \in \mathbb{R}.
\end{equation}
Moreover, we have
\[
\int_{-\infty}^{\infty} W(\Phi_0(z)) \, \mathrm{d}z = \int_{-\infty}^{\infty} \frac{1}{2}|\Phi_0'(z)|^2 \, \mathrm{d}z = \frac{\sqrt{2}}{3}.
\]

Next, we consider the leading order $\mathcal{O}(\varepsilon^{-2})$ in \eqref{temp_equation} and we obtain that
\begin{align*}
    \partial_z \big(k(\Phi_0(z))\partial_z \Theta_0\big) &=0.
\end{align*}
Integrating the equation once we obtain that $k(\Phi_0(z))\partial_z \Theta_0=\text{const.}$
Using the matching condition \eqref{eq: matching 2} implies then that
\begin{align}
    k(\Phi_0(z))\partial_z \Theta_0=0
\end{align}
and from the interface condition $\Theta_0(st,z=0)=\theta_c$ it follows that 
\begin{align}\label{eq: inner temperature condition}
    \Theta_0(t,s,z)=\theta_c.
\end{align}

Next, we take a look at the expansion to order $\mathcal{O}(\epsilon^{-1})$.
For equation \eqref{phi_equation} we obtain
\begin{align}
    -\mathcal{V}\Phi_0'= -\kappa \Phi_0' -W''(\Phi_0)\Phi_1-\frac{\lambda}{\theta_c}(\Theta_0-\theta_c)(\Phi_0^2-1) +\partial_{zz}\Phi_1.
\end{align}
We multiply this by $\Phi_0'(z)$ and integrate with respect to $z$ from $-\infty$ to $\infty$.
Using the equipartition of energy \eqref{eq: equipartition of energy} and \eqref{eq: inner temperature condition}, this then yields
\begin{align*}
    \int_{-\infty}^\infty \big(\kappa-\mathcal{V}\big)|\Phi_0'|^2\,\textnormal{d}z&=  \int_{-\infty}^\infty \partial_{zz}\Phi_1 \Phi'_0- (W'(\Phi_0))'\Phi_1\,\textnormal{d}z\\
    &= \big[W'(\Phi_0)\Phi_1-\partial_z\Phi_1\Phi_0'\big]_{-\infty}^{\infty} + \int_{-\infty}^\infty \partial_z\big( W'(\Phi_0)-\Phi_0''\big)\,\textnormal{d}z =0,
\end{align*}
where the last equality follows from applying the matching conditions and the leading order expansion.
This then yields
\begin{align}
    \mathcal{V}=\kappa,
\end{align}
which means that the motion of the interface is driven by its mean curvature.

Lastly, for equation \eqref{temp_equation} we have in the order $\mathcal{O}(\epsilon^{-1})$ that
\begin{align*}
   - c_v\mathcal{V} \partial_z \Theta_0&= \partial_z\big(k(\Phi_0)\partial_z\Theta_1\big)-\kappa k(\phi)\partial_z \Theta_0-\frac{\lambda}{\theta_c} (\Phi_0^2-1)\Theta_0 \mathcal{V}\Phi_0' +|\mathcal{V}\Phi_0'|^2.
\end{align*}
We can reduce this by applying \eqref{eq: inner temperature condition} and obtain
\begin{align*}
    \partial_z\big(k(\Phi_0)\partial_z\Theta_1\big))=\lambda (\Phi_0^2-1) \mathcal{V}\Phi_0' -|\mathcal{V}\Phi_0'|^2
\end{align*}
Integrating with respect to $z$ from $-\infty$ to $\infty$ and applying the matching condition \eqref{eq: matching 3} together with \eqref{eq: equipartition of energy} yields
\begin{align}
    \big[k(\phi)\nabla \theta_0\cdot \nu\big]_-^+= \mathcal{V}(\lambda \sqrt{2}-\mathcal{V})\frac{2\sqrt{2}}{3}.
\end{align}

\subsection{Sharp interface limit}

By performing a matched asymptotic expansion of the original system, we formally derive the following sharp-interface limit:
\begin{subequations}\label{NLS_1}
\begin{align} 
        c_v\partial_t \theta &=  k_{+1} \Delta \theta~ \textnormal{ in } \Omega_+,\label{eq: limit syst 1}\\
        c_v\partial_t \theta &=  k_{-1} \Delta \theta~\textnormal{ in } \Omega_-,\label{eq: limit syst 2}\\
       \big[\nabla \theta \cdot \nu\big]_-^+&= \mathcal{V}(\lambda \sqrt{2}-\mathcal{V})\frac{2\sqrt{2}}{3}~ \textnormal{ on } \Gamma_0\label{eq: limit syst 3}, 
\end{align}
\end{subequations}
together with the Dirichlet boundary condition for the temperature at the interface:
\begin{equation}\label{NLS_2}
     \theta =\theta_c~ \textnormal{ on } \Gamma_0.
\end{equation}
Here, $\mathcal{V}$ is the the normal velocity at the interface $\Gamma_0$. The formal asymptotic analysis suggests that the interface evolves according to motion by mean curvature:
\begin{align*}
    \mathcal{V} = \kappa\ ,
\end{align*}
where $\kappa$ is the mean curvature of $\Gamma_0$. However, a rigorous mathematical analysis is still required to establish the existence of such a special solution to the Stefan-type problem \eqref{NLS_1} with the Dirichlet boundary condition \eqref{NLS_2}.

\begin{remark}
   We note that the limit system is indeed a two-phase Stefan type problem.
   Hence, our result is in accordance with the observations by Caginalp.
   Thus, we can compare the limit problem obtained from the non-isothermal Allen-Cahn system with the classical two-phase Stefan problem, which reads as follows
    \begin{align}
        \partial_t \theta &= \frac{k_s}{c_s\rho_s}\Delta \theta, \quad \textnormal{in } \Omega_s,\\
        \partial_t \theta &= \frac{k_l}{c_l\rho_l}\Delta \theta, \quad \textnormal{in } \Omega_l,
    \end{align}
    where $k_s, k_l$ are the heat conductivity in the solid and liquid phase, respectively, $c_s, c_l$ denote the specific heat capacities, and $\rho_s, \rho_l$ the corresponding densities.
    To close the system the Stefan condition, a conservation law on the free boundary which balances the heat transported into the free boundary and the melting heat generated through solidification, is introduced.
    It reads as
    \begin{align}
       \mathcal{V}(L_l-L_s)&= -\left[\frac{k_{s/l}}{\rho_{s/l}}\partial_\nu \theta \right]_-^+,\quad \textnormal{on} \Gamma_0
    \end{align}
    where $L$ is the latent heat per unit volume in each of the phases.
\end{remark}

\begin{remark}
    We want to point out that the quadratic nonlinearity in the jump condition \eqref{eq: limit syst 3} is due to the entropy production in the system and therefore a direct consequence of the consistency of the system with the laws of thermodynamic.
\end{remark}

\begin{remark}
   We want to emphasize that the scaling of the melting/freezing energy is critical such that the free interfaces and temperature motion via motion by mean curvature coincide. If the scaling is on the same order as the double-well potential, i.e. we have $\chi(\epsilon)=\epsilon^{-1}$, there will be the discrepancy between the level set of temperature and interface. In this case, the dynamics of temperature is driven by the dissipation, instead of mean curvature, as illustrate in the simulation results in Fig. \ref{Fig5} below.
  %  the diffusion of the temperature drives the interface motion.[NOT exactly] 
\end{remark}

\section{Numerics}\label{sect: Numerics}

In this section, we perform a numerical study of the non-isothermal Allen-Cahn equation (\ref{phi_equation}) - (\ref{temp_equation}) with $\lambda = 1$, which can be rewritten as
\begin{equation}\label{Form_1}
  \begin{cases}
    & \phi_t = - \frac{1}{\epsilon} \mu, \quad \mu = \frac{1}{\epsilon} (\phi^2 - 1) \phi - \epsilon \Delta \phi + \frac{1}{\theta_c} (\theta - \theta_c) (\phi^2 - 1). \\
    & c_v \theta_t =  \nabla \cdot ( \kappa (\phi) \nabla \theta) -  \frac{1} {\epsilon}  \frac{\theta}{\theta_c} (\phi^2 - 1) \mu + \frac{1}{\epsilon} |\mu|^2 \\
  \end{cases}
 \end{equation}
The numerical simulations focus on the effects of the curvature in the dynamics. Our aim is to validate the results of the asymptotic expansion discussed in the previous section and provide the behavior of solution for finite $\epsilon$.

It is noticed that if we rescale the temperature by a constant $\theta_0$, the temperature equation becomes 
\begin{equation}
c_v \theta_0 \tilde{\theta}_t =  \theta_0 \nabla \cdot ( \kappa  \nabla \tilde{\theta}) -  \frac{1} {\tilde{\theta}_c \epsilon} \tilde{\theta} (\phi^2 - 1) \mu + \frac{1}{\epsilon} |\mu|^2 \ , 
\end{equation}
where $\tilde{\theta} = \theta / \theta_0$ and $\tilde{\theta}_c = \theta_c / \theta_0$. 
In the current study, without loss of generality, we take $c_v  = \kappa = 1$, and set $\theta_c = 273.15$. 

To solve the equation (\ref{Form_1}), we adopt the following semi-implicit scheme for temporal discretization:
\begin{itemize}[leftmargin=*, topsep=0pt]
\item Fixed $\theta^{n}$, solve equation of $\phi$ using the implicit Euler scheme
\begin{equation}
   \frac{\phi^{n+1} - \phi^n}{\Delta t} = - \frac{1}{\epsilon} \left( \frac{1}{\epsilon} ( (\phi^{n+1})^2 - 1) \phi^{n+1} - \epsilon \Delta \phi^{n+1} + \frac{1}{\theta_c} (\theta^n - \theta_c) ( (\phi^{n+1})^2 - 1) \right)  
\end{equation}

\item With $\phi^{n+1}$, solve the heat equation using
\begin{equation}
  c_v \frac{\theta^{n+1} - \theta^n}{\Delta t} =  \nabla \cdot ( \kappa \nabla \theta^{n+1})  - \frac{1}{\theta_c \epsilon} \mu^{n+1} g'(\phi^{n+1}) \theta^{n+1} + \frac{1}{\epsilon} |\mu^{n+1}|^2\ ,
\end{equation}
where $$\mu^{n+1} = \frac{1}{\epsilon} ( (\phi^{n+1})^2 - 1) \phi^{n+1} - \epsilon \Delta \phi^{n+1} + \frac{1}{\theta_c} (\theta^n - \theta_c) ( (\phi^{n+1})^2 - 1).$$
\end{itemize}
We use the standard central finite difference scheme for the spatial discretization. 
The current numerical scheme is quite simple and may not preserve the original energetic variational structure in the discrete level.
We plan on developing a structure-preserving numerical scheme to (\ref{Form_1}) in the future work. In all numerical simulations below, we take $\epsilon = 0.05$, $h = 1/50$ and $\Delta t = 10^{-3}$ unless stated otherwise.
 
\begin{figure}[!htb]
%[!htbp]
\centering
\includegraphics[width = 0.7 \linewidth]{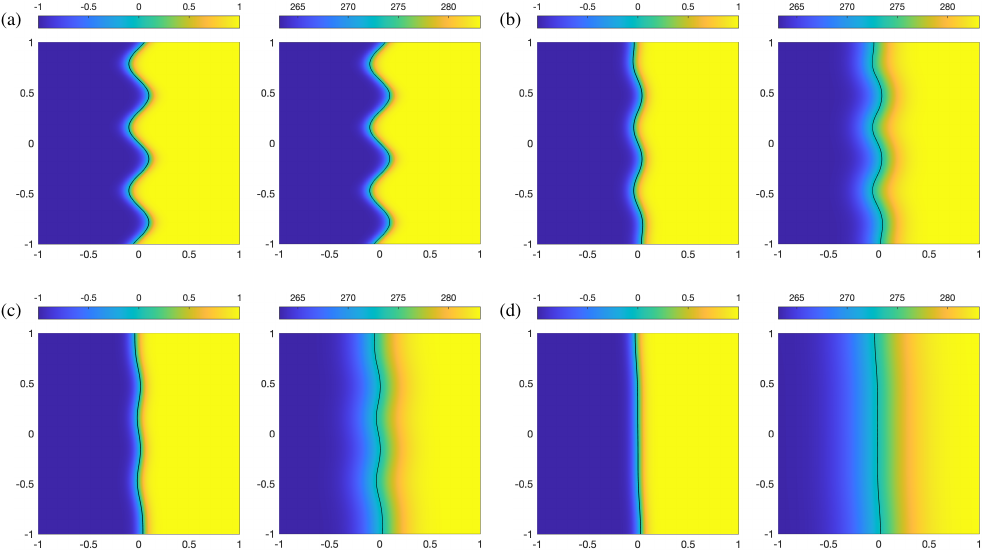}
\caption{Numerical results for $\epsilon = 0.05$ at different time: (a) $t = 0$, (b) $t = 0.01$, (c) $t = 0.02$, and (d) $t = 0.05$.  }\label{Fig_1}
\end{figure}

\subsection{Quasi-1D case}

We first consider a quasi-1D case, by taking the initial condition as $$\phi_0(x, y) = \tanh\left( \frac{x + 0.1 \sin (10 y) }{\sqrt{2} \epsilon} \right), \quad \theta_0 = \theta_c + 10 \phi_0(x, y)$$
We impose the Neumann boundary condition for both $\phi$ and $\theta$. Fig. \ref{Fig_1} shows the numerical results for different times $t$. One can notice that the movement of the interface $\theta = \theta_c$ is almost same to the movement of the interface $\phi = 0$, both are driven by the curvature effect.

The initial condition is motivated by the numerical example in \cite{chen1997simple}, which can be used to explore the Mullins-Sekerka instability in some Stefan problems.
It is worth mentioning that the Allen-Cahn type phase field model mainly captures the curvature effects in the interfacial dynamics. To investigate the Mullins-Sekerka instability and dendritic growth, one should study a non-isothermal Cahn-Hilliard type phase-field model \cite{de2024temperature}.

\subsection{Motion by mean curvature}
%In this subsection, we study the dynamics of interfaces in non-isothermal Allen-Cahn equation. 

Next, we consider the initial conditions  
\begin{equation}\label{mc_in_1}
\phi_0 (\x) = \tanh \left( \frac{\sqrt{x^2 + y^2 - 0.6}}{\sqrt{2}\,\epsilon} \right), \qquad  \theta_0(\x) = 10\,\phi_0(\x) + 273.15,
\end{equation}
where Dirichlet boundary conditions are imposed on both $\phi$ and $\theta$.  

\begin{figure}[!h]
\centering

\begin{subfigure}[b]{0.49\linewidth}
  \includegraphics[width=\linewidth]{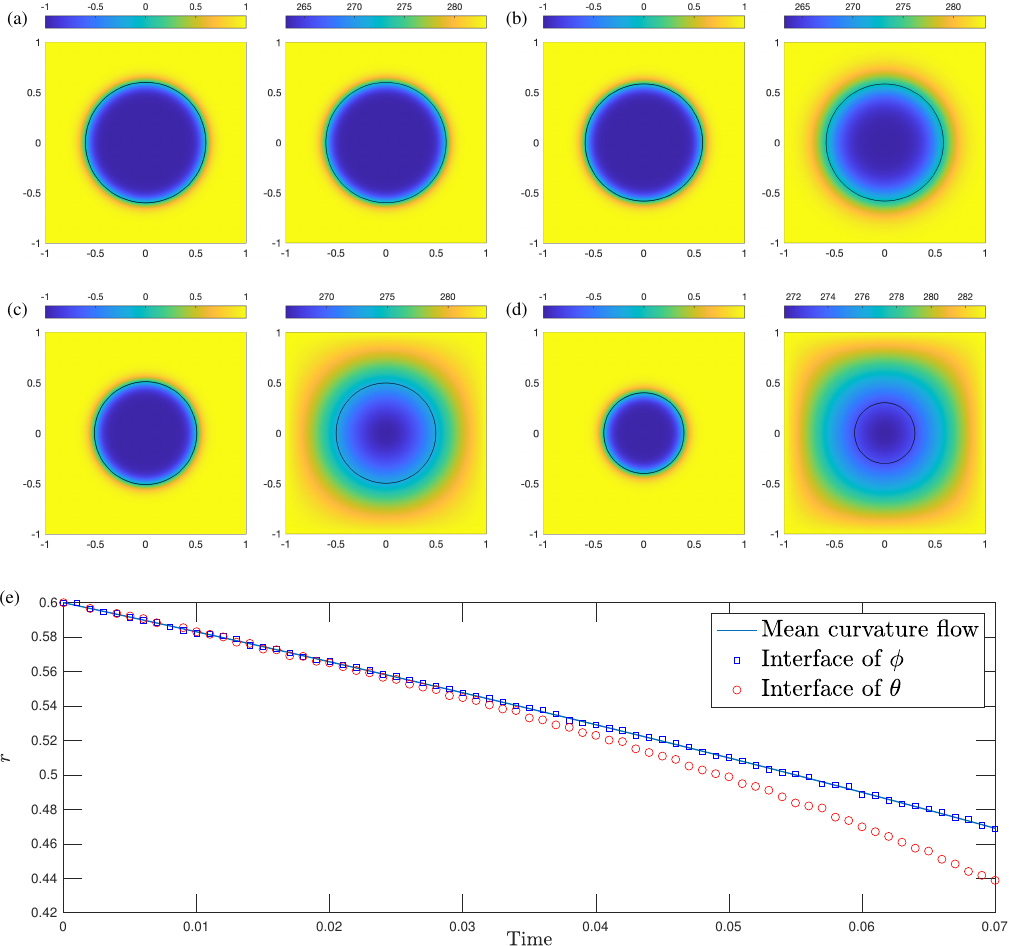}
  \caption{\small Initial condition (\ref{mc_in_1})}
  \label{fig:2a}
\end{subfigure}
\hfill
\begin{subfigure}[b]{0.49\linewidth}
  \includegraphics[width=\linewidth]{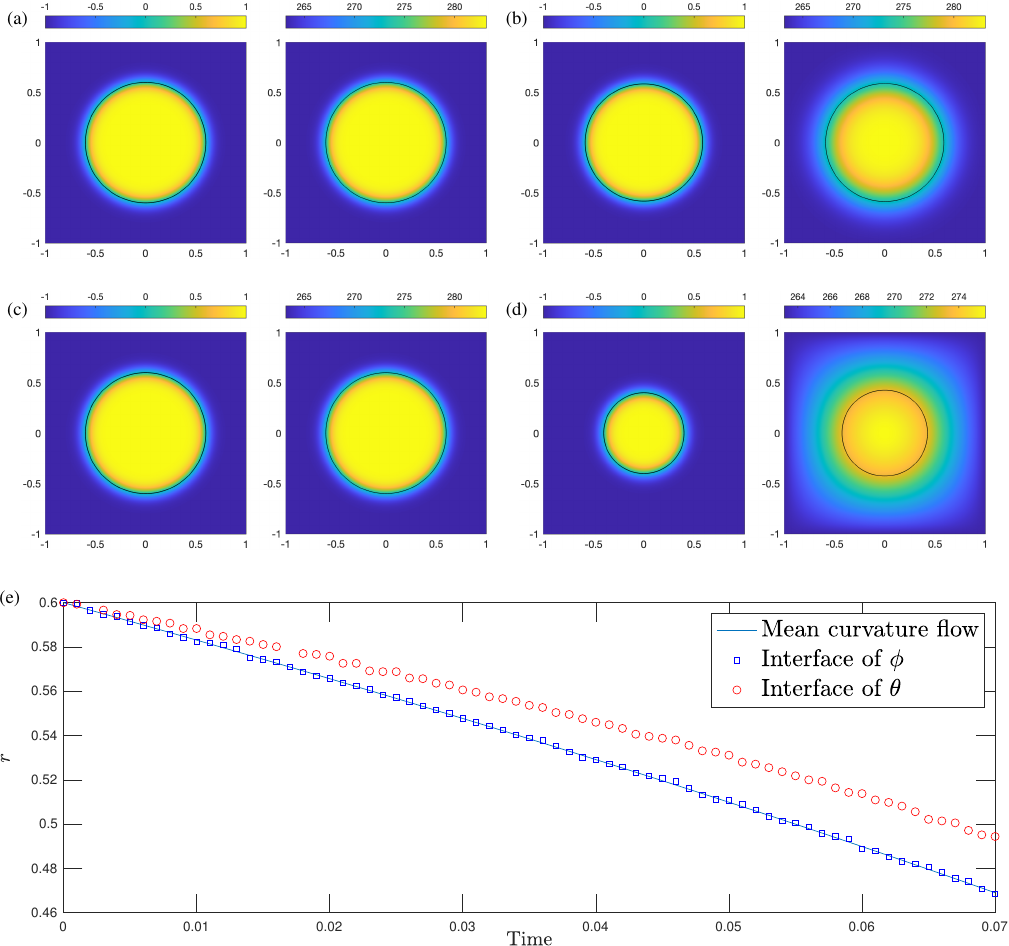}
  \caption{\small Initial condition (\ref{mc_in_2})}
  \label{fig:2b}
\end{subfigure}

\vspace{-1.5 em}
\caption{Numerical results for $\epsilon = 0.05$ under two different initial conditions. Panel~(A) corresponds to initial condition~\eqref{mc_in_1}, and panel~(B) to~\eqref{mc_in_2}. Within each panel, subplots~(a)–(d) show the evolution of $\phi$ and $\theta$ at $t=0,\,0.01,\,0.05,\,0.1$, while subplot~(e) displays the interface positions $\phi = 0$ and $\theta = \theta_c$ as functions of time}
\end{figure}

Figure~\ref{fig:2a}(a)--(d) displays the numerical results for $\epsilon = 0.05$ at different times, while Figure~\ref{fig:2a}(e) shows the evolution of the interfaces defined by $\phi = 0$ and $\theta = \theta_c$.  The asymptotic analysis presented earlier shows that, in the sharp-interface limit, the motion of both interfaces converges to mean curvature flow, described by  
\begin{align*}
R(t) = \sqrt{0.36 - 2t},
\end{align*}
as $\epsilon \to 0$, where $R(t)$ denotes the radius of the interface. Numerical results confirm that, for small times $t$, the interfaces $\phi = 0$ and $\theta = \theta_c$ both evolve according to mean curvature flow, in agreement with the asymptotic prediction. For larger times, however, the two interfaces gradually diverge, exhibiting a \emph{superheating} effect in which the temperature within certain regions of the solid phase exceeds the phase-transition threshold. We also consider the initial condition  
\begin{equation}\label{mc_in_2}
\phi_0 (\x) = -\tanh \left( \frac{\sqrt{x^2 + y^2 - 0.6}}{\sqrt{2}\,\epsilon} \right), 
\qquad 
\theta_0(\x) = 10\,\phi_0(\x) + 273.15,
\end{equation}
with results shown in Fig.~\ref{fig:2b}. Similarly, the $\phi$-interface follows mean curvature flow, while the $\theta = \theta_c$ interface deviates from it. Superheating is again observed, as the temperature in certain regions of the solid phase remains above the critical temperature.

To further demonstrate the numerical convergence of the phase-field model to its sharp-interface limit, we consider four different $\epsilon$: $\epsilon = 0.1, 0.05, 0.02$ and $0.01$. Figures~\ref{Fig_diff_ep}(a)-(b) show the evolution of interfaces for $\phi = 0$ and $\theta = \theta_c$ for different $\epsilon$, respectively. As expected, the $\phi = 0$ interface approaches the mean curvature flow as $\epsilon \to 0$, which is consistent with the asymptotic analysis.
The evolution of the interface $\theta = \theta_c$ is in agreement with mean curvature flow in the short time, but deviate from $\phi = 0$ even for small $\epsilon$ when $t$ goes large, leading to the superheating phenomenon. The gap between the two interfaces does not visibly shrink as $\epsilon$ decreases.

\begin{figure}[!h]
\centering
\begin{overpic}[width = 0.49 \linewidth]{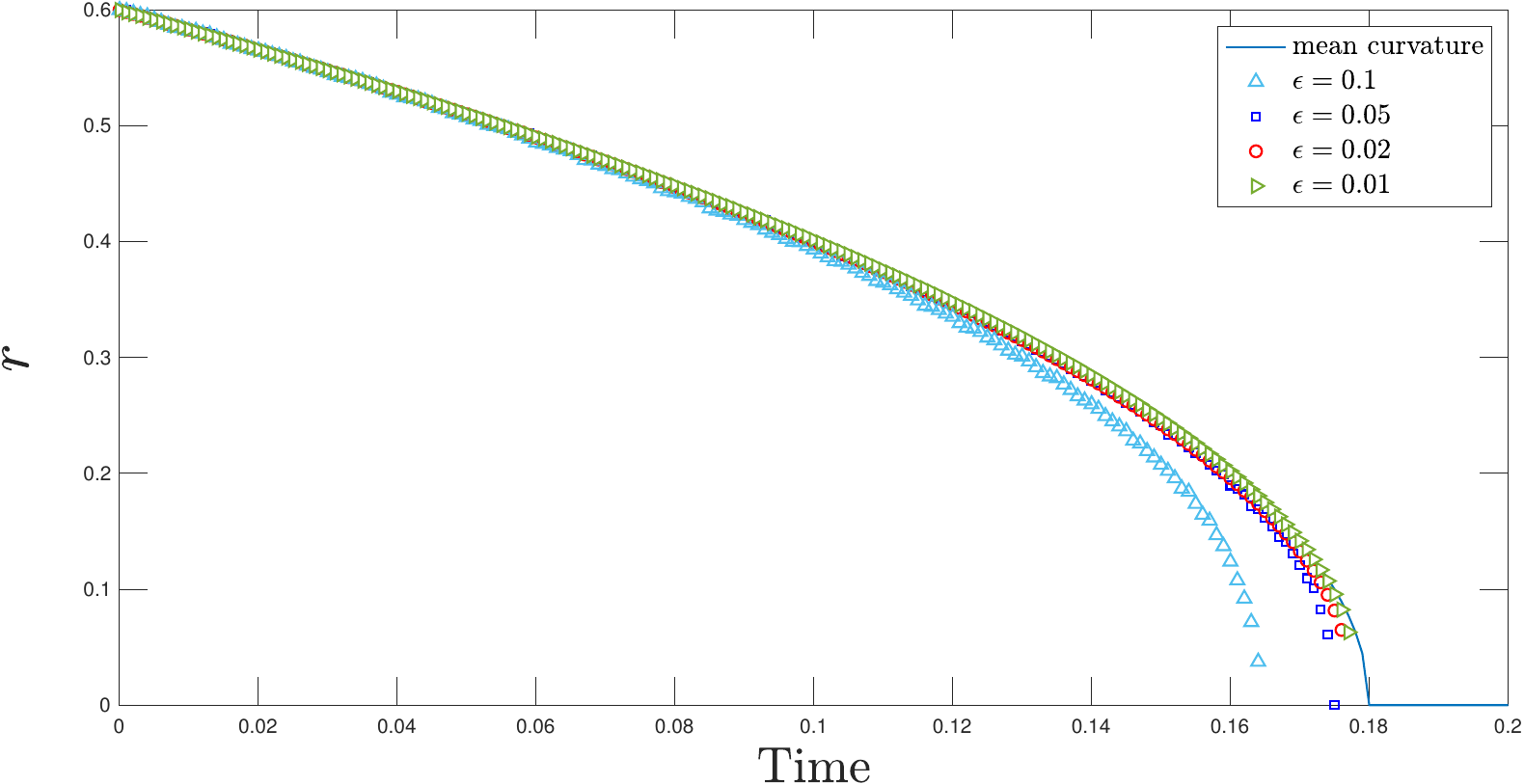}
\put(-3, 45){{(a)}}
\end{overpic}
\hfill
\begin{overpic}
[width = 0.49 \linewidth]{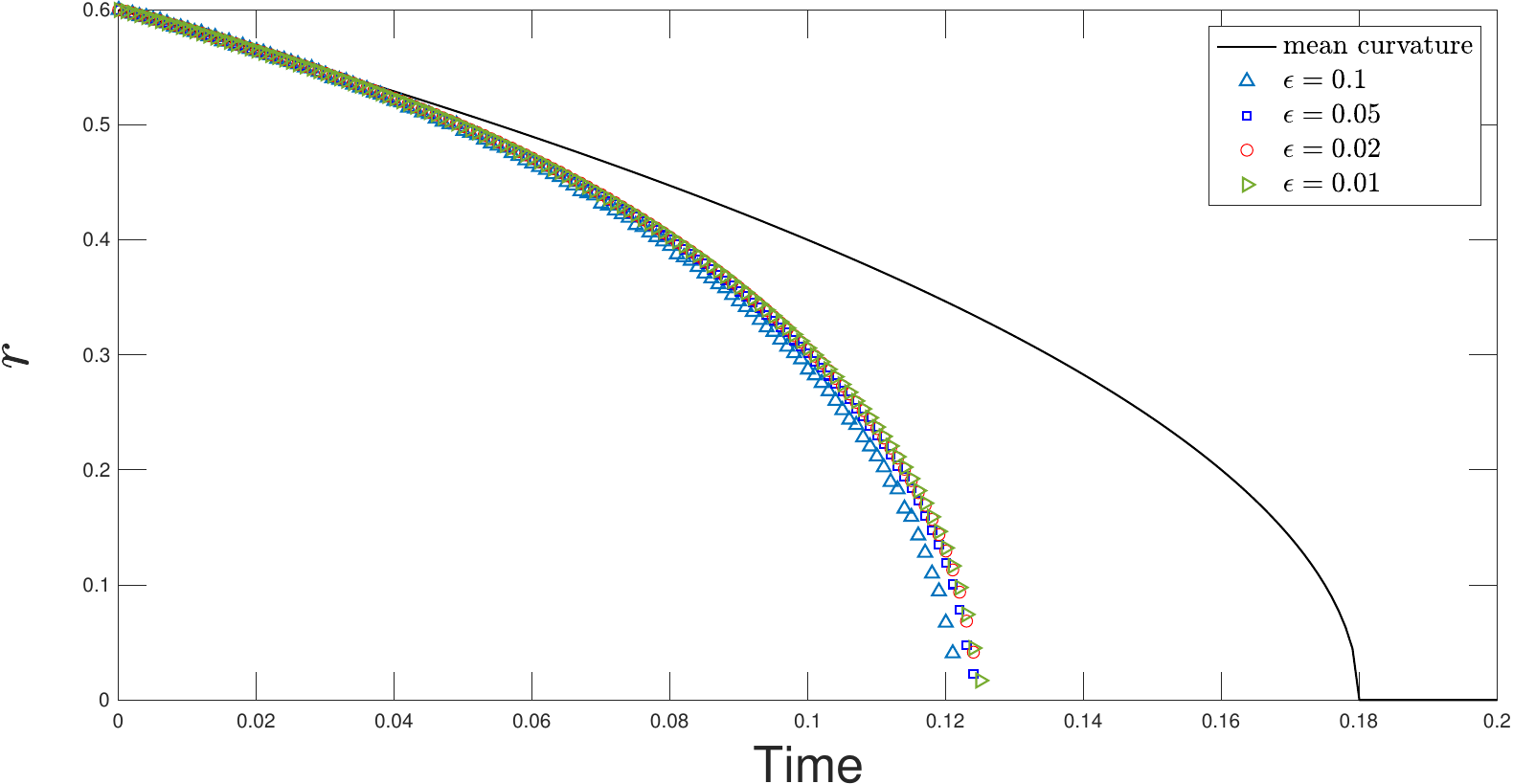}
\put(-3, 45){{\small (b)}}
\end{overpic}
\caption{Evolution of the interfaces: (a) $\phi = 0$ and (b) $\theta = \theta_c$, for varying $\epsilon$.}\label{Fig_diff_ep}
\end{figure}

The numerical result can be understood as follows: the phase-field equation at $\mathcal{O}(1/\epsilon)$ enforces the Dirichlet condition $\theta = \theta_c$ on the moving interface. In principle, as $\epsilon \to 0$, the two interfaces are expected to coincide. Nevertheless, the evolution of the temperature field is governed by a heat equation at $\mathcal{O}(1)$, which adjusts only diffusively in time and is not strictly slaved to the motion of the $\phi$-interface. Consequently, the $\theta = \theta_c$ level set may propagate faster than the $\phi=0$ interface, leading to the superheating effect. This scale separation between the phase-field dynamics and the heat equation likely explains why the numerical simulations do not exhibit a clear collapse of the gap between the two interfaces, even for relatively small $\epsilon$. In practice, resolving this discrepancy would require significantly smaller values of $\epsilon$, which are beyond the range currently tractable by the numerical methods employed.

Next, we consider more complicated initial shapes for the droplets. Fig. \ref{fig:ellipse} shows the numerical results for the initial condition 
\begin{align}\label{eq: ellipsoide initial cond}
    \phi_0 (\x) = \tanh \left( \frac{\sqrt{x^2 + 4y^2 - 0.6} }{\sqrt{2} \epsilon} \right), \quad \theta_0(\x) = 10 \phi_0(\x) + 273.15\ ,
\end{align}
along with Dirichlet boundary conditions. Again, the movement of both interfaces are driven by the mean curve, and almost same initially. But the interface $\theta = \theta_c$ moves faster and disappears early.
The evolution at different snap shots in time is shown in Fig \ref{fig:ellipse} (a)-(f).
We also consider an initial condition with a triangular shape of the droplet, shown in Fig. \ref{fig:tri}. The simulation results are shown in Fig. \ref{fig:tri}(a) - (f) for different times $t$.  Similar to the previous case, the movement of the two interfaces is in agreement when $t$ is small, but will diverge from each other for larger times.

\begin{figure}[!h]
\centering
    \begin{subfigure}[b]{0.48\linewidth}
        \centering
        \includegraphics[width=\linewidth]{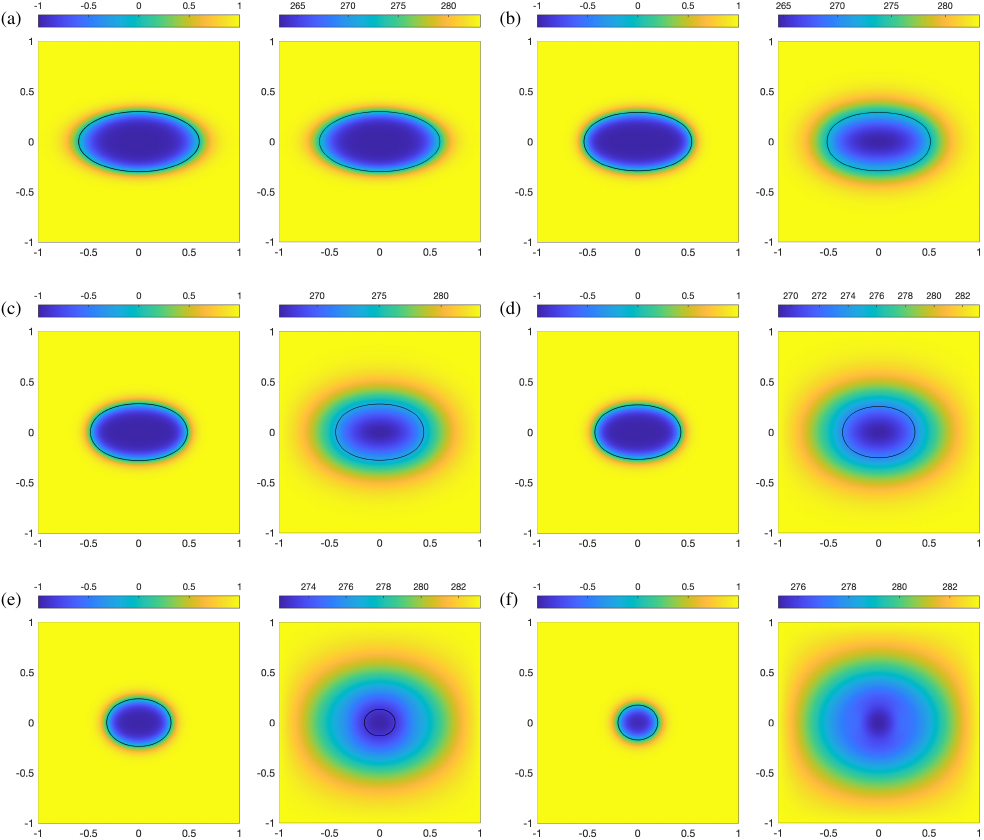}
        \caption{Evolution of initial condition \eqref{eq: ellipsoide initial cond}}
        \label{fig:ellipse}
    \end{subfigure}
    \hfill
    \begin{subfigure}[b]{0.48\linewidth}
        \centering
        \includegraphics[width=\linewidth]{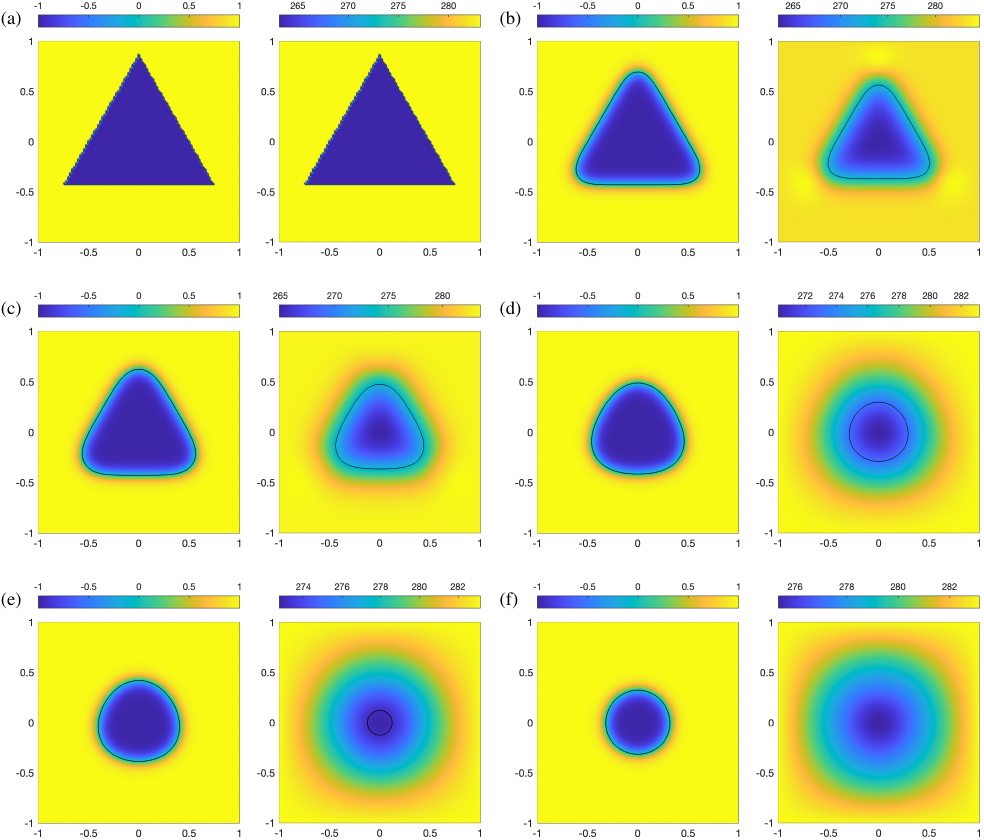}
        \caption{Evolution of triangular initial condition}
        \label{fig:tri}
    \end{subfigure}
\caption{Numerical results for $\epsilon = 0.05$ under two different initial conditions. Panel~(A) corresponds to initial condition~\eqref{mc_in_1}, and panel~(B) to~\eqref{mc_in_2}. Within each panel, subplots~(a)–(d) show the evolution of $\phi$ and $\theta$ at $t=0,\,0.01,\,0.05,\,0.1$, while subplot~(e) displays the interface positions $\phi = 0$ and $\theta = \theta_c$ as functions of time}
\label{Fig3}
\end{figure}

\subsection{Beyond the Stefan problem}

As mentioned earlier, in order to reduce the non-isothermal Allen-Cahn equation to the classical Stefan problem in the sharp interface limit, it is crucial to choose the proper scaling of $\chi(\epsilon)$ in the asymptotic analysis. Fig. \ref{Fig5} shows a simulation result for $\chi(\epsilon) = \frac{1}{\epsilon}$. The initial and boundary conditions are same to that in Fig. \ref{fig:2a}. Under this scaling, the movement of interfaces of the temperature field and phase function are no longer consistent with each other even in short time. Although the curvature effects still dominate the motion of the interface of $\phi$, the movement of the interface is slower than the mean curvature flow. In contrast to the Stefan problem, the movement of interface of the temperature no longer follows the mean curvature flow. 
\begin{figure}[h]
\centering
\includegraphics[width = 0.6 \linewidth]{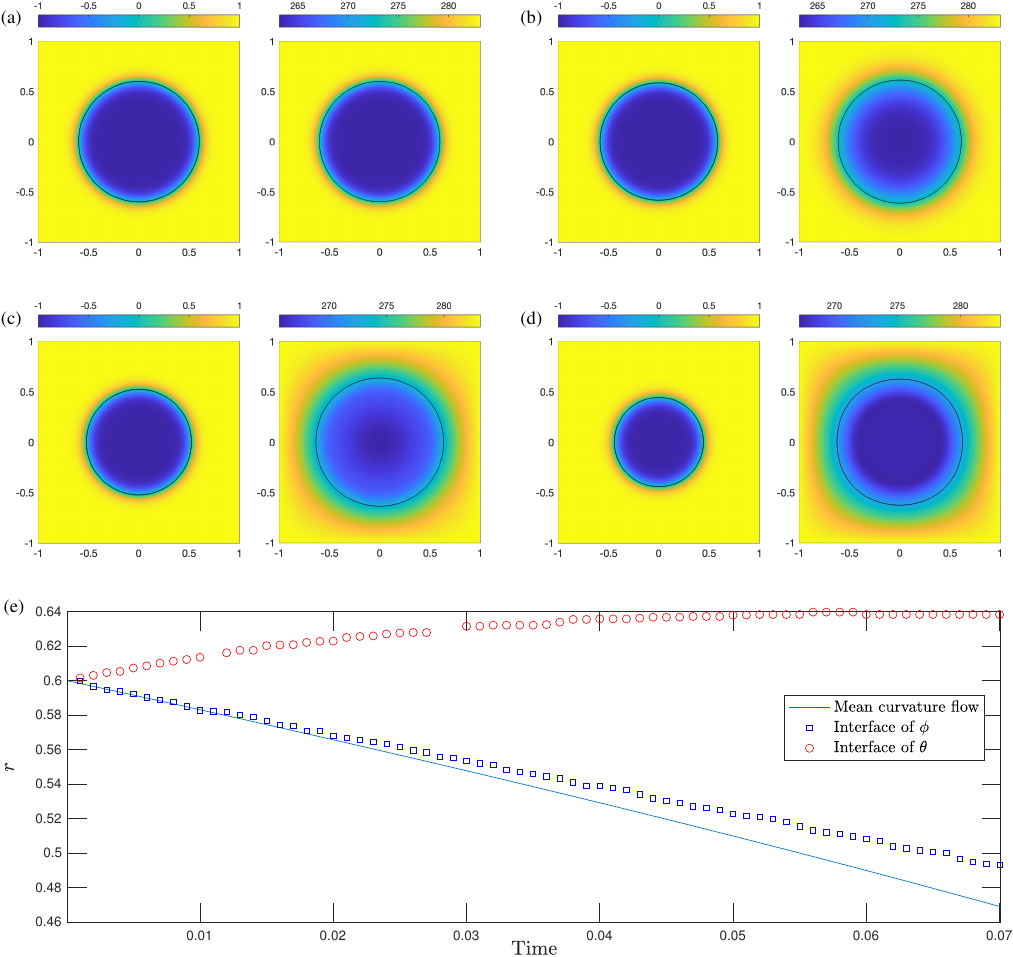}
\caption{Numerical results for $\epsilon = 0.05$ at different time (left: $\phi$  ; right: $\theta$ ): (a) $t = 0$, (b) $t = 0.01$, (c) $t = 0.05$, and (d) $t = 0.1$.  (e) Location of interface $\phi = 0$ and $\theta = \theta_c$ as a function of time. }\label{Fig5}
\end{figure}

\section{Conclusion}
In this paper, we derive a non-isothermal Allen-Cahn model for phase transition and heat transfer by using an energetic variational approach. 
In contrast to the Stefan problem, which employs a single interface to describe phase transition and temperature, the non-isothermal Allen-Cahn model contains more physical information. 
By asymptotic analysis, we have shown that the sharp interface limit of the non-isothermal Allen-Cahn equation, under certain scaling of the melting/freezing energy, results in a nonlinear Stefan type problem.
Various numerical simulations further demonstrate that the Stefan problem is a good approximation to the non-isothermal Allen-Cahn equation in the short time, as evidenced by the close agreement in the evolution of two interfaces. However, over longer time periods, the movement of two interfaces will diverge from each other.

While diffuse interface models can be used to study different effects of interface dynamics during phase transitions, the Allen-Cahn type model examined in this paper mainly focuses on curvature effects. In contrast, other physical phenomena, other physical phenomena, such as the Mullins-Sekerka instability and dendritic growth, can only be captured by the Cahn-Hilliard type diffuse interface model \cite{de2024temperature}. We will perform numerical studies and analyze the asymptotic limit of these diffuse interface models in future work.

\section*{Acknowledgements}
C. L. was partially supported by NSF DMS-2153029, DMS-2118181 and DMS-2410742.
J.-E. S. was supported by the German Science Foundation (DFG) under grant No.456754695.
Y. W. was supported by NSF DMS-2153029 and DMS-2410740.

\bibliographystyle{siam}
\bibliography{lit}

\begin{thebibliography}{10}

\bibitem{alexiades2018mathematical}
{\sc V.~Alexiades}, {\em Mathematical modeling of melting and freezing
  processes}, Routledge, 2018.

\bibitem{allen1979microscopic}
{\sc S.~M. Allen and J.~W. Cahn}, {\em A microscopic theory for antiphase
  boundary motion and its application to antiphase domain coarsening}, Acta
  metallurgica, 27 (1979), pp.~1085--1095.

\bibitem{athanasopoulos1996regularity}
{\sc I.~Athanasopoulos, L.~Caffarelli, and S.~Salsa}, {\em Regularity of the
  free boundary in parabolic phase-transition problems}, Acta Mathematica, 176
  (1996), pp.~245--282.

\bibitem{baker2022zero}
{\sc G.~Baker and M.~Shkolnikov}, {\em Zero kinetic undercooling limit in the
  supercooled stefan problem}, in Annales de l'Institut Henri Poincare (B)
  Probabilites et statistiques, vol.~58, Institut Henri Poincar{\'e}, 2022,
  pp.~861--871.

\bibitem{bernauer2011optimal}
{\sc M.~K. Bernauer and R.~Herzog}, {\em Optimal control of the classical
  two-phase stefan problem in level set formulation}, SIAM Journal on
  Scientific Computing, 33 (2011), pp.~342--363.

\bibitem{blowey1993curvature}
{\sc J.~F. Blowey and C.~M. Elliott}, {\em Curvature dependent phase boundary
  motion and parabolic double obstacle problems}, in Degenerate diffusions,
  Springer, 1993, pp.~19--60.

\bibitem{bronsard1991motion}
{\sc L.~Bronsard and R.~V. Kohn}, {\em Motion by mean curvature as the singular
  limit of ginzburg-landau dynamics}, Journal of differential equations, 90
  (1991), pp.~211--237.

\bibitem{caffarelli1977regularity}
{\sc L.~A. Caffarelli}, {\em The regularity of free boundaries in higher
  dimensions}, Acta Mathematica, 139 (1977), pp.~155--184.

\bibitem{caffarelli1983continuity}
{\sc L.~A. Caffarelli and L.~C. Evans}, {\em Continuity of the temperature in
  the two-phase stefan problem}, Archive for Rational Mechanics and Analysis,
  81 (1983), pp.~199--220.

\bibitem{caffarelli1979continuity}
{\sc L.~A. Caffarelli and A.~Friedman}, {\em Continuity of the temperature in
  the stefan problem}, Indiana University Mathematics Journal, 28 (1979),
  pp.~53--70.

\bibitem{caginalp1986analysis}
{\sc G.~Caginalp}, {\em An analysis of a phase field model of a free boundary},
  Archive for rational mechanics and analysis, 92 (1986), pp.~205--245.

\bibitem{caginalp1990dynamics}
\leavevmode\vrule height 2pt depth -1.6pt width 23pt, {\em The dynamics of a
  conserved phase field system: Stefan-like, hele-shaw, and cahn-hilliard
  models as asymptotic limits}, IMA Journal of Applied Mathematics (Institute
  of Mathematics and Its Applications), 44 (1990), pp.~77--77.

\bibitem{caginalp1998convergence}
{\sc G.~Caginalp and X.~Chen}, {\em Convergence of the phase field model to its
  sharp interface limits}, European Journal of Applied Mathematics, 9 (1998),
  pp.~417--445.

\bibitem{caginalp1995derivation}
{\sc G.~Caginalp and J.~Jones}, {\em A derivation and analysis of phase field
  models of thermal alloys}, Annals of physics, 237 (1995), pp.~66--107.

\bibitem{cahn1958free}
{\sc J.~W. Cahn and J.~E. Hilliard}, {\em Free energy of a nonuniform system.
  i. interfacial free energy}, The Journal of chemical physics, 28 (1958),
  pp.~258--267.

\bibitem{chen1997simple}
{\sc S.~Chen, B.~Merriman, S.~Osher, and P.~Smereka}, {\em A simple level set
  method for solving stefan problems}, Journal of Computational Physics, 135
  (1997), pp.~8--29.

\bibitem{ciarletta2011phase}
{\sc M.~Ciarletta, M.~Fabrizio, and V.~Tibullo}, {\em A phase field model for a
  solid--liquid phase transition}, Mechanics Research Communications, 38
  (2011), pp.~477--480.

\bibitem{coleman1967thermodynamics}
{\sc B.~D. Coleman and M.~E. Gurtin}, {\em Thermodynamics with internal state
  variables}, The journal of chemical physics, 47 (1967), pp.~597--613.

\bibitem{colli1995penrose}
{\sc P.~Colli and J.~Sprekels}, {\em On a penrose-fife model with zero
  interfacial energy leading to a phase-field system of relaxed stefan type},
  Annali di Matematica Pura ed Applicata, 169 (1995), pp.~269--289.

\bibitem{colli1997stefan}
\leavevmode\vrule height 2pt depth -1.6pt width 23pt, {\em Stefan problems and
  the Penrose-Fife phase field model}, vol.~7, 1997.

\bibitem{cuchiero2023propagation}
{\sc C.~Cuchiero, S.~Rigger, and S.~Svaluto-Ferro}, {\em Propagation of
  minimality in the supercooled stefan problem}, The Annals of Applied
  Probability, 33 (2023), pp.~1588--1618.

\bibitem{de2019non}
{\sc F.~De~Anna and C.~Liu}, {\em Non-isothermal general ericksen--leslie
  system: derivation, analysis and thermodynamic consistency}, Archive for
  Rational Mechanics and Analysis, 231 (2019), pp.~637--717.

\bibitem{de2024temperature}
{\sc F.~De~Anna, C.~Liu, A.~Schl{\"o}merkemper, and J.-E. Sulzbach}, {\em
  Temperature dependent extensions of the cahn--hilliard equation}, Nonlinear
  Analysis: Real World Applications, 77 (2024), p.~104056.

\bibitem{dewynne1990survey}
{\sc J.~N. Dewynne}, {\em A survey of supercooled stefan problems}, in
  Mini-Conference on Free and Moving Boundary and Diffusion Problems (Canberra,
  1990), vol.~30, 1990, pp.~42--56.

\bibitem{escher1996analyticity}
{\sc J.~Escher and G.~Simonett}, {\em Analyticity of the interface in a free
  boundary problem}, Mathematische Annalen, 305 (1996), pp.~439--460.

\bibitem{Fremond}
{\sc M.~Fr{\'e}mond}, {\em Non-Smooth Thermomechanics}, Springer-Verlag Berlin
  Heidelberg, 2002.

\bibitem{garcke2016cahn}
{\sc H.~Garcke, K.~F. Lam, E.~Sitka, and V.~Styles}, {\em A
  cahn--hilliard--darcy model for tumour growth with chemotaxis and active
  transport}, Mathematical Models and Methods in Applied Sciences, 26 (2016),
  pp.~1095--1148.

\bibitem{garcke2014diffuse}
{\sc H.~Garcke, K.~F. Lam, and B.~Stinner}, {\em Diffuse interface modelling of
  soluble surfactants in two-phase flow}, Communications in Mathematical
  Sciences, 12 (2014), pp.~1475--1522.

\bibitem{giga2017variational}
{\sc M.-H. Giga, A.~Kirshtein, and C.~Liu}, {\em Variational modeling and
  complex fluids}, Handbook of mathematical analysis in mechanics of viscous
  fluids,  (2017), pp.~1--41.

\bibitem{ilmanen1993convergence}
{\sc T.~Ilmanen}, {\em Convergence of the allen-cahn equation to brakke's
  motion by mean curvature}, Journal of Differential Geometry, 38 (1993),
  pp.~417--461.

\bibitem{kinderlehrer1978smoothness}
{\sc D.~Kinderlehrer and L.~Nirenberg}, {\em The smoothness of the free
  boundary in the one phase stefan problem}, Communications on Pure and Applied
  Mathematics, 31 (1978), pp.~257--282.

\bibitem{lame1831memoire}
{\sc G.~Lam{\'e} and B.~Clapeyron}, {\em M{\'e}moire sur la solidification par
  refroidissement d’un globe liquide}, in Annales Chimie Physique, vol.~47,
  1831, pp.~250--256.

\bibitem{Landau2013collected}
{\sc L.~Landau}, {\em Collected papers of L.D. Landau}, Elsevier, 2013.

\bibitem{lasarzik2022weak}
{\sc R.~Lasarzik, E.~Rocca, and G.~Schimperna}, {\em Weak solutions and
  weak-strong uniqueness for a thermodynamically consistent phase-field model},
  Atti Accad. Naz. Lincei Cl. Sci. Fis. Mat. Natur., 33 (2022), pp.~229--269.

\bibitem{Liu2022}
{\sc C.~Liu and J.-E. Sulzbach}, {\em The brinkman-fourier system with ideal
  gas equilibrium}, Discrete and Continuous Dynamical Systems, 42 (2022),
  pp.~425--462.

\bibitem{liu2022well}
{\sc C.~Liu and J.-E. Sulzbach}, {\em Well-posedness for the reaction-diffusion
  equation with temperature in a critical besov space}, Journal of Differential
  Equations, 325 (2022), pp.~119--149.

\bibitem{liu2020variational}
{\sc C.~Liu and Y.~Wang}, {\em A variational lagrangian scheme for a
  phase-field model: A discrete energetic variational approach}, SIAM Journal
  on Scientific Computing, 42 (2020), pp.~B1541--B1569.

\bibitem{mamode2013two}
{\sc M.~Mamode}, {\em Two phase stefan problem with boundary temperature
  conditions: An analytical approach}, SIAM Journal on Applied Mathematics, 73
  (2013), pp.~460--474.

\bibitem{mariano2022sources}
{\sc P.~M. Mariano and M.~Spadini}, {\em Sources of finite speed temperature
  propagation}, Journal of Non-Equilibrium Thermodynamics, 47 (2022),
  pp.~165--178.

\bibitem{mcfadden1993phase}
{\sc G.~McFadden, A.~Wheeler, R.~Braun, S.~Coriell, and R.~Sekerka}, {\em
  Phase-field models for anisotropic interfaces}, Physical Review E, 48 (1993),
  p.~2016.

\bibitem{miranville2005nonisothermal}
{\sc A.~Miranville and G.~Schimperna}, {\em Nonisothermal phase separation
  based on a microforce balance}, Discrete and Continuous Dynamical Systems
  Series B, 5 (2005), p.~753.

\bibitem{pegler2021convective}
{\sc S.~S. Pegler and M.~S.~D. Wykes}, {\em The convective stefan problem:
  Shaping under natural convection}, Journal of Fluid Mechanics, 915 (2021),
  p.~A86.

\bibitem{penrose1990thermodynamically}
{\sc O.~Penrose and P.~C. Fife}, {\em Thermodynamically consistent models of
  phase-field type for the kinetic of phase transitions}, Physica D: Nonlinear
  Phenomena, 43 (1990), pp.~44--62.

\bibitem{pruss2007existence}
{\sc J.~Pruss, J.~Saal, and G.~Simonett}, {\em Existence of analytic solutions
  for the classical stefan problem}, Mathematische Annalen, 338 (2007),
  pp.~703--756.

\bibitem{rubinshteuin1971stefan}
{\sc L.~Rubinshte{\u\i}n}, {\em The stefan problem}, vol.~27, American
  Mathematical Soc., 1971.

\bibitem{sherman1970general}
{\sc B.~Sherman}, {\em A general one-phase stefan problem}, Quarterly of
  Applied Mathematics, 28 (1970), pp.~377--382.

\bibitem{stefan1891theorie}
{\sc J.~Stefan}, {\em {\"U}ber die theorie der eisbildung, insbesondere
  {\"u}ber die eisbildung im polarmeere}, Annalen der Physik, 278 (1891),
  pp.~269--286.

\bibitem{sun2024thermodynamically}
{\sc Y.~Sun, J.~Wu, M.~Jiang, S.~M. Wise, and Z.~Guo}, {\em A thermodynamically
  consistent phase-field model and an entropy stable numerical method for
  simulating two-phase flows with thermocapillary effects}, arXiv preprint
  arXiv:2404.04950,  (2024).

\bibitem{visintin2008introduction}
{\sc A.~Visintin}, {\em Introduction to stefan-type problems}, Handbook of
  differential equations: evolutionary equations, 4 (2008), pp.~377--484.

\bibitem{wang2022some}
{\sc Y.~Wang and C.~Liu}, {\em Some recent advances in energetic variational
  approaches}, Entropy, 24 (2022), p.~721.

\end{thebibliography}

\end{document}